\documentclass[11pt]{amsart}

\makeatletter
\usepackage{amssymb}
\usepackage{latexsym}
\usepackage{amsbsy}
\usepackage{amsfonts}
\usepackage{dsfont,txfonts
}

\newtheorem{theorem}[equation]{Theorem}
\newtheorem{proposition}[equation]{Proposition}
\newtheorem{lemma}[equation]{Lemma}
\newtheorem{corollary}[equation]{Corollary}
\newtheorem{definition}[equation]{Definition}

\theoremstyle{definition}
\newtheorem{remark}[equation]{Remark}

\newtheorem{example}[equation]{Example}

\numberwithin{equation}{section}

\newdimen\AAdi%
\newbox\AAbo%
\def\AAk#1#2{\setbox\AAbo=\hbox{#2}\AAdi=\wd\AAbo\kern#1\AAdi{}}%

\def\eqref#1{(\ref{#1})}
\def\eqlabel#1{\def\@currentlabel{#1}}

\def\formula#1{\def\@tempa{#1}\let\@tempb\theequation\def\theequation{%
\hbox{#1}}\def\@currentlabel{(\theequation)}$$}
\def\endformula{\leqno\hbox{(\@tempa)}$$\@ignoretrue\let\theequation\@tempb}

\def\given{\hskip5\p@\relax\vrule\@width.4\p@\hskip5\p@\relax}

\newcommand{\open}[1]{%
\par\normalfont\topsep6\p@\@plus6\p@\trivlist\item[\hskip\labelsep\itshape#1%
\@addpunct{.}]\ignorespaces}

\DeclareRobustCommand{\close}[1]{%
  \ifmmode 
  \else \leavevmode\unskip\penalty9999 \hbox{}\nobreak\hfill
  \fi
  \quad\hbox{$#1$}}

\newlength{\toskip}

\def \R {{\mathbb R}}

\def \L {{\mathbb L}}

\def \Var {\textrm{Var}}

\newcommand{\bF}{\overline{F}_\mu}

\newcommand{\Osc}{\mathrm{Osc}}

\makeatother

\begin{document}
\date{\today}

\title[Functional inequalities for heavy tails distributions]{Functional inequalities for heavy tails distributions and application to isoperimetry}

\author[P. Cattiaux]{\textbf{Patrick Cattiaux}}
\address{{\bf {Patrick} CATTIAUX}\\
Institut de Math\'ematiques de Toulouse, CNRS UMR 5219 \\
Universit\'e Paul Sabatier, Laboratoire de Statistique et Probabilit\'es\\
118 route de Narbonne, F-31062 Toulouse cedex 09, FRANCE.}
\email{patrick.cattiaux@math.univ-toulouse.fr}

\author[N. Gozlan]{\textbf{\; Nathael Gozlan}}
\address{{\bf {Nathael} GOZLAN}\\
Laboratoire d'Analyse et Math\'ematiques Appliqu\'ees- UMR 8050,
Universit\'es de Paris-est, Marne la
Vall\'ee et de Paris 12-Val-de-Marne\\
Boulevard Descartes, Cit\'e Descartes, Champs sur Marne\\
77454 Marne la Vall\'ee Cedex 2, FRANCE.}
\email{nathael.gozlan@univ-mlv.fr}

\author[A. Guillin]{\textbf{\; Arnaud Guillin}}
\address{{\bf {Arnaud} GUILLIN}\\
Ecole Centrale Marseille et LATP \, Universit\'e  de Provence,
Technopole Ch\^ateau-Gombert, 39, rue F. Joliot Curie, 13453 Marseille Cedex 13, FRANCE.}
\email{aguillin@ec-marseille.fr, guillin@cmi.univ-mrs.fr}

\author[C. Roberto]{\textbf{\; Cyril Roberto}}
\address{{\bf {Cyril} ROBERTO}\\
Laboratoire d'Analyse et Math\'ematiques Appliqu\'ees- UMR 8050,
Universit\'es de Paris-est, Marne la
Vall\'ee et de Paris 12-Val-de-Marne\\
Boulevard Descartes, Cit\'e Descartes, Champs sur Marne\\
77454 Marne la Vall\'ee Cedex 2, FRANCE.}
\email{cyril.roberto@univ-mlv.fr}

%
%
%

\begin{abstract}
This paper is devoted to the study of probability measures with heavy tails.
Using the Lyapunov function approach we prove that such measures satisfy different kind of functional inequalities
such as weak Poincar\'e and weak Cheeger, weighted Poincar\'e and weighted Cheeger inequalities and their dual forms.
Proofs are short and we cover very large situations.
For product measures on $\R^n$ we obtain the optimal dimension dependence using the mass transportation method.
Then we derive (optimal) isoperimetric inequalities.
Finally we deal with spherically symmetric measures.
We recover and improve many previous results.
\end{abstract}

\maketitle

%
%


\bigskip

\textit{ MSC 2000 : .}
\bigskip

\section{Introduction, definitions and first results.}\label{intro}

The subject of functional inequalities knows an amazing growth due to the numerous fields of
application: differential geometry, analysis of p.d.e., concentration of measure phenomenon,
isoperimetry, trends to equilibrium in deterministic and stochastic evolutions... Let us mention Poincar\'e,
weak Poincar\'e or super Poincar\'e inequalities, Sobolev like inequalities, $F$-Sobolev
inequalities (in particular the logarithmic Sobolev inequality), modified log-Sobolev inequalities
and so on. Each type of inequality appears to be very well adapted to the study of one (or more) of the applications listed above.
We refer to \cite{grossbook},
\cite{bakry}, \cite{Dav}, \cite{logsob}, \cite{ledoux}, \cite{Wbook}, \cite{GZ99}, \cite{Ro99},
\cite{BCR1}, \cite{CGWW} for an introduction.

If a lot of results are known for log-concave probability measures, not so much has been proved
for measures with heavy tails (let us mention \cite{r-w,BCR2, BLweight,BCG,bob07,bobkov-zeg2}). In
this paper the focus is on such measures with heavy tails and our aim is to prove functional and
isoperimetric inequalities.

Informally measures with heavy tails are measures with tails larger than exponential. Particularly
interesting classes of examples are either $\kappa$-concave probability measures, or
sub-exponential like laws (or tensor products of any of them) defined as follows.
\smallskip

We say that a probability measure $\mu$ is $\kappa$-concave with $\kappa= - 1/\alpha$ if
\begin{equation}\label{eqconvex}
d\mu(x) = V(x)^{-(n+\alpha)}dx
\end{equation}
with $V : \mathbb{R}^n \to (0,\infty)$ convex and $\alpha > 0$. Such measures have been introduced by Borell \cite{borell75} in more general setting.
See \cite{bob07} for a comprehensive introduction and the more general definition of $\kappa$-concave probability measures.
Prototypes of $\kappa$-concave probability measures are
the generalized Cauchy distributions
\begin{equation}\label{eqcauchy}
d\mu(x) = \frac{1}{Z} \, \left((1+|x|^2)^{1/2}\right)^{-(n+\alpha)}
\end{equation}
for $\alpha > 0$, which corresponds to the previous description since $x \mapsto (1+|x|^2)^{1/2}$
is convex. In some situations we shall also consider
$d\mu(x)=(1/Z)\left((1+|x|)\right)^{-(n+\alpha)}$. Note that these measures are Barenblatt
solutions in porous medium equations and appears naturally in weighted porous medium equations,
giving the decay rate of this nonlinear semigroup towards the equilibrium measure, see
\cite{vasquez,DGGW}.

\smallskip

We may replace the power by an exponential yielding the notion of sub-exponential law, i.e. given
any convex function $V : \mathbb{R}^n \to (0,\infty)$ and $p
> 0$, we shall say that
\begin{equation*}
d\mu(x) = e^{-V(x)^{p}}dx
\end{equation*}
is a sub-exponential like law. A typical example is $V(x)=|x|$.
\smallskip

Heavy tails measures are now particularly important since they appear in various areas: fluid
mechanics, mathematical physics, statistical mechanics, mathematical finance ... Since previous
results in the literature are not optimal, our main goal is to study the isoperimetric problem for
heavy tails measures. This will lead us to consider various functional inequalities (weak Cheeger,
weighted Cheeger, converse weighted Cheeger). Let us explain why.
\medskip

Recall the isoperimetric problem.

Denote by $d$ the Euclidean distance on $\mathbb{R}^n$. For $h\ge 0$ the closed $h$-enlargement of a set
$A\subset \mathbb{R}^n$ is $A_h:=\big\{x\in M;\; d(x,A)\le h\big\}$
where $d(x,A):=\inf\{ d(x,a);\; a \in A\}$ is $+\infty $ by convention for $A=\emptyset$.
We may define the boundary measure, in the sense of $\mu$, of a Borel set $A \subset \mathbb{R}^n$ by
$$
\mu_s(\partial A):= \liminf_{h \to 0^+} \frac{\mu(A_h\setminus A)}{h}\cdot
$$
An isoperimetric inequality is of the form
\begin{equation} \label{eq:iniso}
\mu_s(\partial A) \geq F( \mu(A)) \qquad \forall A \subset \mathbb{R}^n
\end{equation}
for some function $F$. Their study is an important topic
in geometry, see e.g. \cite{ros,barthe05}. The first question of interest is to find the optimal $F$. Then one can try to
find the optimal sets for which \eqref{eq:iniso} is an equality. In general this is very difficult and the only hope is
to estimate the isoperimetric profile defined by
$$
I_\mu(a):=\inf\big\{ \mu_s(\partial A);\; \mu(A)=a\big\}, \quad a\in[0,1].
$$
Note that the isoperimetric inequality \eqref{eq:iniso} is closely related to concentration of
measure phenomenon, see \cite{bobkh97cbis,ledoux01}. For a large class of distributions $\mu$ on
the line with exponential or faster decay, it is possible to prove
\cite{borell,sudakov,bobk96ffii,BH97,BL96,BCR1,BCR3,Emil} that the isoperimetric profile
$I_{\mu^n}$ of the $n$-tensor product $\mu^n$ is (up a to universal, hence dimension free
constants) equal to $I_\mu$.

For measures with heavy tails, this is no more true. Indeed,
 if $\mu$ is a probability measure on $\mathbb{R}$ such that
there exist $h>0$ and $\varepsilon>0$ such that for all $n\ge 1$ and
all $A\subset \mathbb{R}^{n}$ with $\mu^{n}(A)\ge \frac12$, one has
\begin{equation} \label{eq:talagrand}
\mu^{n}(A+[-h,h]^{n})\ge\frac12+\varepsilon,
\end{equation}
then $\mu$ has exponential tails, that is there exist positive constants $C_1,C_2$
such that $\mu([x,+\infty))\le C_1 e^{-C_2 x}$,
$x\in\mathbb{R}$, see \cite{Tal91}.

Therefore, for measures with heavy tails, the isoperimetric profile as well as the concentration
of measure for product measure should heavily depend on $n$. Some bounds on $I_{\mu^n}$, not
optimal in $n$, are obtained in \cite{BCR2} using weak Poincar\'e inequality. The non optimality
is mainly due to the fact that $\mathbb{L}_2$ inequalities (namely weak Poincar\'e inequalities)
are used. We shall obtain optimal bounds, thus completing the pictures for the isoperimetric
profile of tensor product of very general form of probability measures, using $\mathbb{L}_1$
inequalities called weak Cheeger inequalities we introduce now.

As noted by Bobkov \cite{bob07}, for measures with heavy tails, isoperimetric inequalities are
equivalent to weak Cheeger inequalities. A probability measure is said to satisfy a weak Cheeger
inequality if there exists some non-increasing function $\beta : (0,\infty) \to [0,\infty)$ such
that for every smooth $f : \mathbb{R}^n \to \mathbb{R}$, it holds
\begin{equation} \label{in:wc}
\int \, |f-m| \, d\mu \, \leq \, \beta(s)  \int \, |\nabla f| \, d\mu \, + \, s \, \Osc_\mu(f)\, \qquad \forall s>0 ,
\end{equation}
where $m$ is a median of $f$ for $\mu$ and $\Osc_\mu(f)={\rm ess \, sup}(f) - {\rm ess \,
inf}(f)$. The relationship between $\beta$ in \eqref{in:wc} and $F$ in \eqref{eq:iniso} is
explained in Lemma \ref{lembob} below. Since $\int \, |f-m| \, d\mu \leq \frac 12\Osc_\mu(f)$,
only the values $s \in (0,1/2]$ are relevant.

Recall that similar weak Poincar\'e inequalities were introduced in \cite{r-w}, replacing the
median by the mean and introducing squares.

Of course if $\beta(0)<+\infty$ we recover the usual Cheeger or Poincar\'e inequalities.
\smallskip

In order to get isoperimetric results, we thus investigate such inequalities. We use two main
strategies. One is based on the Lyapunov function approach \cite{BCG,CGWW,BBCG}, the other is
based on mass transportation method \cite{gozjfa,gozlan} (see also
\cite{bobkov94,Tal94,BH97,BH00}). In the first case proofs are very short. The price to pay is a
rather poor control on the constants, in particular in terms of the dimension. But we cover very
general situations (not at all limited to $\kappa$-concave like measures). The second strategy
gives very explicit controls on the constants, but results are limited to tensor products of
measures on the line or spherically symmetric measures (but only for the $\mathbb L_2$ case).

This is not surprising in view of the analogue results known for log-concave measures for
instance. Indeed recall that the famous conjecture of Kannan-Lovasz-Simonovits (\cite{KLS})
telling that the Poincar\'e constant of log-concave probability measures only depends on their
variance is still a conjecture. In this situation universal equivalence between Cheeger's
inequality and Poincar\'e inequality is known (\cite{ledgap,Emil}), and some particular cases (for
instance spherically symmetric measures) have been studied (\cite{bobsphere}). In our situation
the equivalence between weak Poincar\'e and weak Cheeger inequalities does not seem to be true in
general, so our results are in a sense the natural extension of the state of the art to the heavy
tails situation.
\medskip

The Lyapunov function approach appears to be a very powerful tool not only when dealing with the $\mathbb{L}_1$  form
\eqref{in:wc} but also with $\mathbb{L}_2$ inequalities.

This approach is well known for dynamical systems for example. It has been introduced by
Khasminski and developed by Meyn and Tweedie (\cite{MT,MT2,MT3}) in the context of Monte Carlo
algorithm (Markov chains). This dynamical approach is in some sense natural: consider the process
whose generator is symmetric with respect to the studied measure (see next section for more
precise definitions), Lyapunov conditions express that there is some drift (whose strength varies
depending on the measure studied) which pushes the process to some natural, say compact, region of
the state space. Once in the compact the process behaves nicely and pushed forward to it as soon
as it escapes. It is then natural that it gives nice qualitative (but not so quantitative) proofs
of total variation convergence of the associated semigroup towards its invariant measure and find
applications in the study of the decay to equilibrium of dynamical systems, see e.g.
\cite{DFG,HairerM,V97,BCG,CatGui3}. It is also widely studied in statistics, see e.g. \cite{MT}
and the references therein. In \cite{BCG}, connections are given between Lyapunov functions and
functional inequalities of weak Poincar\'e type, improving some existing criteria discussed in
\cite{r-w,BCR2}. In this paper we give new types of Lyapunov functions (in the spirit of
\cite{BBCG}) leading to quantitative improvements and in some sense optimal results. Actually we
obtain four types of functional inequalities: weighted Cheeger (and weighted Poincar\'e
inequalities)
\begin{equation} \label{in:wc1}
\int \, |f-m| \, d\mu \, \leq \, C  \, \int \, |\nabla f| \, \omega \, d\mu \,
\end{equation}
and their dual forms called converse Cheeger (and converse Poincar\'e inequalities)
\begin{equation} \label{in:wc2}
\inf_c \, \int \, |f-c| \, \omega \, d\mu \, \leq \, C \,  \int \, |\nabla f| \, d\mu
\end{equation}
where $\omega$ are suitable ``weights'' (see Section \ref{sec2} for precise and more general
definitions definitions).

Weighted Cheeger and weighted Poincar\'e inequalities were very recently studied by Bobkov and
Ledoux \cite{BLweight}, using functional inequalities of Brascamp-Lieb type. Their results apply
to $\kappa$-concave probability measures. We recover their results with slightly worst constants
but our approach also applies to much general type laws (sub-exponential for example).

Note that converse Poincar\'e inequalities appear in the spectral theory of Schr\"{o}dinger
operators, see \cite{DMC}. We will not pursue this direction here.

Our approach might be summarized by the following diagram:
$$
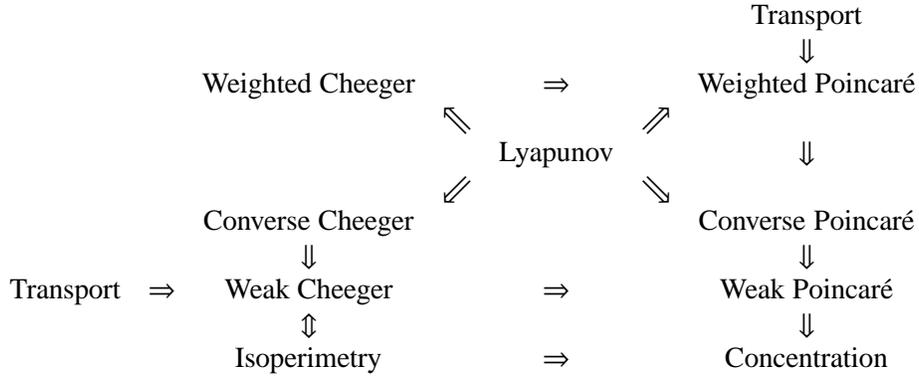
\begin{array}{ccccccc}
 &  & &          &          &          &\textrm{Transport} \\
 &  & &          &          &          &\Downarrow \\
 &  & \textrm{Weighted Cheeger} &          &  \Rightarrow        &          & \textrm{Weighted Poincar\'e} \\
 &  &                  & \Nwarrow &          & \Nearrow & \\
 &  &                  &          & \textrm{Lyapunov} &          &\Downarrow  \\
 &  &                  & \Swarrow &          & \Searrow & \\
 &  & \textrm{Converse Cheeger} &          &          &          & \textrm{Converse Poincar\'e} \\
 &  &  \Downarrow       &          &          &         & \Downarrow \\
\textrm{Transport} & \Rightarrow & \textrm{Weak Cheeger} & & \Rightarrow  & & \textrm{Weak Poincar\'e} \\
 &  & \Updownarrow     &          &          &          & \Downarrow \\
 &  & \textrm{Isoperimetry}     &          & \Rightarrow & & \textrm{Concentration}  \\
\end{array}
$$

\smallskip

Some points have to be underlined. As the diagram indicates, converse inequalities are suitable
for obtaining isoperimetric (or concentration like) results, while (direct) weighted inequalities,
though more natural, are not. Indeed, the tensorization property of the variance immediately shows
that if $\mu$ satisfies a weighted Poincar\'e inequality with constant $C$ and weight $\omega$,
then the tensor product $\mu^n$ satisfies the same inequality. Since we know that the
concentration property for heavy tails measures is not dimension free, this implies that contrary
to the ordinary or the weak Poincar\'e inequality, the weighted Poincar\'e inequality cannot
capture the concentration property of $\mu$. The other point is that the mass transportation
method can also be used to obtain some weighted Poincar\'e inequalities, and weighted Poincar\'e inequalities via a change of function lead to converse Poincar\'e inequality (see \cite{BLweight}). The final point is that on most examples
we obtain sharp weights (but non necessarily sharp constants), showing that (up to constants) our
results are optimal.

\medskip
The paper is organized as follows.

In Section \ref{sec2} we prove that the existence of a Lyapunov function implies  weighted Cheeger
and weighted Poincar\'e inequalities and their converse.

Section \ref{secweak} is devoted to the study of weak Cheeger inequalities and to their
application to the isoperimetric problem. The Lyapunov function approach and the transport
technique are used. Explicit examples are given.

Then, weighted Poincar\'e inequalities are proved in Section \ref{secspherique} for some
spherically  symmetric probability measures with heavy tails. We use there the transport
technique.

We show in Section \ref{sec:linkswp} how to obtain weak Poincar\'e inequalities from weak Cheeger and converse Poincar\'e inequalities.

Finally, the appendix is devoted to the proof of some technical results used in Section \ref{secweak}.


\section{From $\phi$-Lyapunov function to weighted inequalities and their converse}\label{sec2}
The purpose of this section is to derive weighted inequalities of Poincar\'e and Cheeger types,
and their converse forms,  from the existence of a $\phi$ Lyapunov function for the underlying
diffusion operator. To properly define this notion let us describe the general framework we shall
deal with.

Let $E$ be some Polish state space equipped with a probability measure $\mu$ and a $\mu$-symmetric
operator $L$.  The main assumption on $L$ is that there exists some algebra $\mathcal{A}$ of
bounded functions, containing constant functions, which is everywhere dense (in the $\L_2(\mu)$
norm) in the domain of $L$. This ensures the existence of a ``carr\'e du champ'' $\Gamma$, {\it
i.e.} for $f, g \in \mathcal{A}$, $L(fg)=f Lg + g Lf + 2 \Gamma(f,g)$. We also assume that
$\Gamma$ is a derivation (in each component), {\it i.e.} $\Gamma(fg,h)=f\Gamma(g,h) + g
\Gamma(f,h)$. This is the standard ``diffusion'' case in \cite{bakry} and we refer to the
introduction of \cite{cat4} for more details. For simplicity we set $\Gamma(f)=\Gamma(f,f)$. Note
that, since $\Gamma$ is a non-negative bilinear form (see \cite[Proposition 2.5.2]{logsob}), the
Cauchy-Schwarz inequality holds: $\Gamma(f,g) \leq \sqrt{\Gamma(f)} \sqrt{\Gamma(g)}$.
Furthermore,  by symmetry,
\begin{equation} \label{eq:gamma}
\int \Gamma(f,g) d\mu = - \int f \, L g \, d\mu \, .
\end{equation}
Also, since $L$ is a diffusion, the following chain rule formula
$\Gamma(\Psi(f), \Phi(g))=\Psi'(f) \Phi'(g) \Gamma(f,g)$ holds.

In particular if $E=\R^n$, $\mu(dx)= p(x) dx$ and $L=\Delta + \nabla \log p.\nabla$, we may
consider the $C^\infty$ functions with compact support (plus the constant functions) as the
interesting subalgebra $\mathcal{A}$, and then $\Gamma(f,g)=\nabla f \cdot \nabla g$.

Now we define the notion of $\Phi$-Lyapunov function.

\begin{definition}\label{deflyapphi}
Let $W\geq 1$ be a smooth enough function on $E$ and $\phi$ be a $\mathcal{C}^1$ positive increasing function
defined on $\R^+$. We say that $W$ is a $\phi$-Lyapunov function if there exist some set $K \subset E$ and
some $b \geq 0$ such that
$$
LW \, \le \,  -\phi(W) \, + \, b \, \mathds{1}_K \, .
$$
This latter condition is sometimes called a
``drift condition''.
\end{definition}
Note that, for simplicity of the previous definition, we did not (and we shall not) specify the underlying operator $L$.

\begin{remark}
One may ask about the meaning of $LW$ in this definition. In the $\R^n$ case, we shall choose
$C^2$ functions $W$, so that $LW$ is defined in the usual sense. On more general state spaces of
course, the easiest way is to assume that $W$ belongs to the ($\L_2$) domain of $L$, in particular
$LW \in \L_2$. But in some situations one can also relax the latter, provided all calculations
can be justified. \hfill $\diamondsuit$
\end{remark}

\subsection{Weighted Poincar\'e inequality and weighted Cheeger inequality.}\label{secgeneral}
In this section we derive weighted Poincar\'e and weighted Cheeger inequalities from the existence of a $\phi$-Lyapunov function.

\begin{definition}\label{defweight}
We say that $\mu$ satisfies a weighted Cheeger (resp. Poincar\'e) inequality with weight $\omega$
(resp. $\eta$) if for some $C,D >0$ and all $g \in \mathcal{A}$ with $\mu$-median equal to $0$,
\begin{equation}\label{eqwcheeg}
\int \, |g| \, d\mu \, \leq \, C \, \int \, \sqrt{\Gamma (g)} \, \omega \, d\mu \, ,
\end{equation}
respectively, for all $g \in \mathcal{A}$,
\begin{equation}\label{eqwpoinc}
\Var_\mu(g) \, \leq \, D \,  \int \, \Gamma (g) \, \eta \, d\mu \, .
\end{equation}
\end{definition}
The standard method shows that if \eqref{eqwcheeg} holds, then \eqref{eqwpoinc} also holds with
$D=4C^2$ and  $\eta=\omega^2$ (see Corollary \ref{cor}).

In order to deal with the ``local'' part $b \mathds{1}_K$ in the definition of a $\phi$-Lyapunov function,
we shall use the notion of local Poincar\'e inequality we introduce now.

\begin{definition}\label{defpoincloc}
Let $U \subset E$. We shall say that $\mu$ satisfies a local Poincar\'e
inequality on $U$ if there exists some constant $\kappa_U$ such that for all
$f \in \mathcal{A}$
$$
\int_U \, f^2 \, d\mu \, \leq \, \kappa_U \, \int_E \Gamma(f) d\mu \, + \, (1/\mu(U)) \,
\left(\int_U \, f \, d\mu\right)^2 \, .
$$
\end{definition}
Notice that in the right hand side the energy is taken over the whole space $E$ (unlike the usual definition). Moreover,
$\int_U \, f^2 \, d\mu - (1/\mu(U)) \, \left(\int_U \, f \, d\mu\right)^2 = \mu(U)\Var_{\mu_U}(f)$
with $\frac{d\mu_U}{d\mu}:= \frac{\mathds{1}_U}{\mu(U)}$. This justifies the name ``local Poincar\'e inequality''.

Now we state our first general result.
\begin{theorem}[Weighted Poincar\'e inequality]\label{thmwp}
Assume that there exists some $\phi$-Lyapunov function $W\in \mathcal{A}$ (see Definition \ref{deflyapphi}) and
that $\mu$ satisfies a local Poincar\'e inequality on some subset $U\supseteq K$ . Then
for all $g \in \mathcal{A}$, it holds

\begin{equation}\label{eqwp}
\Var_\mu(g) \, \leq \, \max \left( \frac{b \kappa_U}{\phi(1)}, 1 \right)  \int \, \left(1 + \frac{1}{\phi'(W)}\right) \, \Gamma(g) \, d\mu
\, .
\end{equation}
\end{theorem}
\begin{proof}
Let $g \in \mathcal{A}$, choose $c$ such that $\int_U (g-c) d\mu = 0$ and set $f = g-c$.
Since $\Var_\mu(g)=\inf_a \, \int (g-a)^2 \, d\mu$, 
we have
$$
\Var_\mu(g) \leq \int f^2 d\mu \leq \int   \, \frac{-LW}{\phi(W)} \, f^2 \,
d\mu + \int f^2 \, \frac{b}{\phi(W)} \, \mathds{1}_K \, d\mu \, .
$$
To manage the second term, we first use that $\Phi(W) \geq \Phi(1)$. Then, the definition of $c$ and the local Poincar\'e inequality ensures that
\begin{eqnarray*}
\int_K \,  f^2 \, d\mu & \leq & \int_U \,  f^2 \, d\mu \\ & \leq & \kappa_U \, \int_E \Gamma(f)
d\mu \, + \, (1/\mu(U)) \, \left(\int_U \, f \, d\mu\right)^2 \\ & = & \kappa_U \, \int_E
\Gamma(g) d\mu \, .
\end{eqnarray*}
\medskip
For the first term, we use Lemma \ref{lem:philyapounov} below (with $\psi=\phi$ and $h=W$). This ends the proof.
\end{proof}
\smallskip

\begin{lemma} \label{lem:philyapounov}
 Let $\psi : \mathbb{R}^+ \to \mathbb{R}^+$ be a $\mathcal{C}^1$ increasing function. Then, for any $f, h \in \mathcal{A}$,
$$
\int  \, \frac{-Lh}{\psi(h)} \, f^2 \, d\mu  \leq  \int \frac { \, \Gamma(f)}{ \psi'(h)} d\mu
$$
\end{lemma}

\begin{proof}
By \eqref{eq:gamma}, the fact that $\Gamma$ is a derivation and the chain rule formula, we have
\begin{eqnarray*}
\int  \, \frac{-Lh}{\psi(h)} \, f^2 \, d\mu
& = & \int \, \Gamma \left( h, \frac{f^2}{\psi(h)} \right) \, d\mu
=
\int \left(\frac{2 \, f \,
\Gamma(f,h)}{\psi(h)} \, - \, \frac{f^2 \psi'(h) \Gamma(h)}{\psi^2(h)}\right) d\mu \, .
\end{eqnarray*}
Since $\psi$ is increasing and according to Cauchy-Schwarz inequality we get
\begin{eqnarray*}
\frac{f \, \Gamma(f,h)}{\psi(h)}
& \leq &
\frac{f \sqrt{\Gamma(f) \Gamma(h)}}{\psi(h)}
 =
\frac{ \sqrt{\Gamma(f)} }{ \sqrt{\psi'(h)} } \cdot
\frac{f \sqrt{\psi'(h) \Gamma(h)}}{\psi(h)}
\\
&\leq &
\frac{1}{2} \frac{\Gamma(f)}{\psi'(h)} + \frac{1}{2} \, \frac{f^2 \psi'(h)\,
\Gamma(h)}{\psi^2(h)}.
\end{eqnarray*}
The result follows.
\end{proof}

\begin{remark}\label{remtrick}
To be rigorous one has to check some integrability conditions in the previous proof. If $W$
belongs to the domain of $L$, the previous derivation is completely rigorous since we are first
dealing with bounded functions $g$. If we do not have a priori controls on the integrability of
$LW$ (and $\Gamma(f,W)$) one has to be more careful.

In the $\R^n$ case there is no real difficulty provided $K$ is compact and $U$ is for instance a
ball $B(0,R)$. To overcome all difficulties in this case, we may proceed as follows : we  first
assume that $g$ is compactly supported and $f=(g-c) \chi$, where $\chi$ is a non-negative
compactly supported smooth function, such that $\mathds{1}_{U} \leq \chi \leq 1$. All the calculation
above are thus allowed. In the end we choose some sequence $\chi_k$ satisfying $\mathds{1}_{kU} \leq
\chi_k \leq 1$, and such that $|\nabla \chi_k|\leq 1$, and we go to the limit. \hfill
$\diamondsuit$
\end{remark}

\begin{remark}
Very recently, two of the authors and various coauthors have pushed forward the links between Lyapunov functionals (and local inequalities) and usual functional inequalities. for example if $\phi$ (in the Lyapunov condition) is assumed to be linear, then we recover the results in \cite{BBCG}, namely a Poincar\'e inequality (and a short proof of Bobkov's result on logconcave probability measure satisfying spectral gap inequality). If $\phi$ is superlinear, then the authors of \cite{CGWW} have obtained super-Poincar\'e inequalities, including nice alternative proofs of Bakry-Emery or Kusuocka-Stroock criterion for logarithmic Sobolev inequality.
$\diamondsuit$
\end{remark}
\bigskip

The same ideas can be used to derive $\L_1$ weighted Poincar\'e (or weighted Cheeger)
inequalities.

Consider $f$ an arbitrary smooth function with median w.r.t. $\mu$ equal to 0. Assume that $W$ is
a $\phi$-Lyapunov function. Then if $f=g-c$,

\begin{eqnarray*}
\int|f|d\mu&\le&\int |f| \frac{-LW}{\phi(W)}d\mu+b \, \int_K \, \frac{|f|}{\phi(W)} \, d\mu\\
&\le&
\int\Gamma\left( \frac{|f|}{\phi(W)},W\right)+\frac{b}{\phi(1)} \, \int_K \, |f| \, d\mu\\
&\le&
\int \frac{\Gamma(|f|,W)}{\phi(W)}d\mu \,
- \, \int \frac{|f|\Gamma(W)\phi'(W)}{\phi^2(W)}d\mu
+ \frac{b}{\phi(1)} \, \int_K \, |f| d\mu \, .
\end{eqnarray*}
Now we use Cauchy-Schwarz for the first term (i.e. $\Gamma(u,v) \leq \sqrt{\Gamma(u)} \,
\sqrt{\Gamma(v)}$) in the right hand side, we remark that the second term is negative since
$\phi'$ is positive, and we can control the last one as before if we assume a local Cheeger
inequality, instead of a local Poincar\'e inequality. We have thus obtained

\begin{theorem}\label{thmcheeg}
Assume that there exists a $\phi$-Lyapunov function $W$ and $\mu$ satisfies some local Cheeger
inequality
$$
\int_U \, |f| \, d\mu \, \leq \, \kappa_U \, \int_E \sqrt{\Gamma(f)} d\mu \, ,
$$ for some $U\supseteq K$ and
all $f$ with median w.r.t. $\mathds{1}_U \, \mu/ \mu(U)$ equal to $0$. Then
for all $g \in \mathcal{A}$ with median w.r.t. $\mu$ equal to 0, it holds
\begin{equation}\label{eqcheeg}
\int \, |g| \, d\mu \, \leq \, \max \left( \frac{b \kappa_U}{\phi(1)}, 1 \right) \, \int  \, \left(1 +
\frac{\sqrt{\Gamma(W)}}{\phi(W)}\right) \, \sqrt{\Gamma(g)} \, d\mu \, .
\end{equation}
\end{theorem}

Again one has to be a little more careful in the previous proof, with integrability conditions,
but difficulties can be overcome as before.
\medskip

It is well known that Cheeger inequality implies Poincar\'e inequality. This is also true for weighted inequalities:

\begin{corollary}\label{cor}
Under the assumptions of Theorem \ref{thmcheeg}, for all $g \in \mathcal{A}$, it holds
$$
\Var_\mu(g) \,  \leq \, 8 \, \max \left( \frac{b \kappa_U}{\phi(1)}, 1 \right)^2  \,
\int \, \left(1 + \frac{\Gamma(W)}{\phi^2(W)}\right) \, \Gamma(g) \, d\mu \, .
$$

\end{corollary}

\begin{proof}
As suggested in the proof of Theorem 5.1 in \cite{BLweight}, if $g$ has a $\mu$ median equal to 0,
$g_+=\max(g,0)$ and $g_-=\max(-g,0)$ too. We may thus apply Theorem \ref{thmcheeg} to both $g_+^2$ and $g_-^2$,
yielding
$$
\int \, g_+^2 \, d\mu \, \leq \, 2 \max \left( \frac{b \kappa_U}{\phi(1)}, 1 \right) \, \int \, g_+ \, \sqrt{\Gamma(g_+)} \, \left(1 +
\frac{\sqrt{\Gamma(W)}}{\phi(W)}\right) \, d\mu \,
$$
and similarly for $g_-$. Applying Cauchy-Schwarz
inequality, and using the elementary $(a+b)^2 \leq 2 a^2 + 2 b^2$ we get that
$$
\int \, g_+^2 \, d\mu \leq \, 8 \, \max \left( \frac{b \kappa_U}{\phi(1)}, 1 \right)^2  \,
\int \, \left(1 + \frac{\Gamma(W)}{\phi^2(W)}\right) \, \Gamma(g_+) \, d\mu \,
$$
and similarly for $g_-$.
To conclude the proof, it remains to sum-up the positive and the negative parts
and to notice that $\Var_\mu(g) \leq  \int \, g^2 \, d\mu$.
\end{proof}

Note that the forms of weight obtained respectively in Theorem \ref{thmwp} and last corollary are different. But, up to
constant, they are of the same order in all examples we shall treat in the following section.


\subsection{Examples in $\R^n$.}\label{secexample}

We consider here the $\R^n$ situation with $d \mu(x)=p(x) dx$ and $L=\Delta + \nabla \log
p.\nabla$, $p$ being smooth enough. We can thus use the argument explained in remark \ref{remtrick}
so that as soon as $W$ is $C^2$ one may apply Theorem \ref{thmwp} and Theorem \ref{thmcheeg}.

Recall the following elementary lemma whose proof can be found in \cite{BBCG}.

\begin{lemma}\label{lemfranck}
If $V$ is convex and $\int e^{-V(x)} \, dx < +\infty$, then
\begin{itemize} \item[(1)] \quad for all $x$, $x.\nabla V(x) \geq V(x) - V(0)$, \item[(2)] \quad
there exist $\delta >0$ and $R>0$ such that for $|x|\geq R$, $V(x)-V(0) \geq \delta \, |x|$.
\end{itemize}
\end{lemma}
%
%

We shall use this Lemma in the following examples.
Our first example corresponds to the convex case discussed by Bobkov and Ledoux \cite{BLweight}.

\begin{proposition}[Cauchy type law]\label{exconvexe}
Let $d\mu(x) = (V(x))^{-(n+\alpha)} \, dx$ for some positive convex function $V$ and
$\alpha> 0$. Then there exists $C > 0$ such that for all $g$
\begin{gather*}
\Var_\mu(g) \, \leq \, C \int \, |\nabla g(x)|^2 \, (1+|x|^2)
\, d\mu(x),\\
\int \, |g-m| \, d\mu \, \leq \, C \int \,
|\nabla g(x)| \, (1 + |x|) \, d\mu(x),
\end{gather*}
where $m$ stands for a median of $g$ under $\mu$.
\end{proposition}

\begin{remark}
The restriction $\alpha >0$ is the same as in \cite{BLweight}.
\end{remark}

\begin{proof}
By Lemma \ref{lem:convex} below, there exists a $\phi$-Lyapunov function $W$ satisfying
 $(1/\phi'(W))(x)= \frac{k}{c(k-2)} |x|^2$ for $x$ large.
Hence, in order to apply Theorem \ref{thmwp} it remains to recall that since $d\mu/dx$ is bounded from
below and from above on any ball $B(0,R)$, 
 $\mu$ satisfies a Poincar\'e inequality and a
Cheeger inequality on such subset, hence a local Poincar\'e (and Cheeger) inequality in the sense
of definition \ref{defpoincloc} (or Theorem \ref{thmcheeg}). This ends the proof.
\end{proof}

\begin{lemma}\label{lem:convex}
Let $L= \Delta - (n+\alpha) (\nabla V/V) \nabla$ with $V$ and $\alpha$ as in Proposition \ref{exconvexe}. Then,
there exists $k>2$,  $b, R>0$ and $W \geq 1$ such that
$$
LW \leq - \phi(W) + b \mathds{1}_{B(0,R)}
$$
with $\phi(u)=c u^{(k-2)/k}$ for some constant $c>0$.  Furthermore, one can choose
$W(x)=|x|^k$ for $x$ large.
\end{lemma}

\begin{proof}
Let $L= \Delta - (n+\alpha) (\nabla V/V) \nabla$ and choose $W \geq 1$
smooth and satisfying $W(x)=|x|^k$ for $|x|$ large enough and $k>2$ that will be chosen later.
For $|x|$ large enough we have
$$
LW(x) = k \, (W(x))^{\frac{k-2}{k}} \, \left(n+k-2 - \frac{(n+\alpha) \, x.\nabla
V(x)}{V(x)}\right) \, .
$$
Using (1) in Lemma \ref{lemfranck} (since $V^{-(n+\alpha)}$ is integrable $e^{-V}$ is also integrable)  we have
$$
n+k-2 - \frac{(n+\alpha) \,
x.\nabla V(x)}{V(x)} \leq k-2 - \alpha + (n+\alpha) \frac{V(0)}{V(x)} \, .
$$
Using (2) in Lemma
\ref{lemfranck} we see that we can
choose $|x|$ large enough for $\frac{V(0)}{V(x)}$ to be less than $\varepsilon$, say
$|x|>R_\varepsilon$. It remains to choose $k>2$ and $\varepsilon>0$ such that
$$
k + n \varepsilon -2 - \alpha (1-\varepsilon)\leq -\gamma
$$
for some $\gamma>0$. We have shown that, for $|x|>R_\varepsilon$,
$$
LW \leq - k \gamma \phi(W),
$$
with $\phi(u) =u^{\frac{k-2}{k}}$ (which is increasing since $k>2$). A compacity argument achieves the proof.
\end{proof}

\begin{remark}
The previous proof gives a non explicit constant $C$ in terms of $\alpha$ and $n$.
This is mainly due to the fact that we are not able to control properly the local Poincar\'e and Cheeger inequalities on balls
for the general measures $d\mu =(V(x))^{-(n+\alpha)} \, dx$. More could be done on specific laws.
\end{remark}

Our next example deals with sub-exponential distributions.

\begin{proposition}[Sub exponential like law]\label{ex2}
Let $d \mu= (1/Z_p) \, e^{- V^p}$
for some positive convex function $V$ and $p>0$.
Then there exists $C > 0$ such that for all $g$
\begin{gather*}
\Var_{\mu}(g) \, \leq \, C \, \int \, |\nabla g(x)|^2 \,
\left(1+(1+|x|)^{2(1-p)}\right) \, d\mu(x) \, ,\\
\int \, |g-m| \, d\mu \,
\leq \, C \, \int \, |\nabla g(x)| \, \left(1 + (1+|x|)^{(1-p)}\right) \,
d\mu(x) \, ,
\end{gather*}
where $m$ stands for a median of $g$ under $\mu$.
\end{proposition}
\begin{remark}
For $p<1$ we get some weighted inequalities, while for $p \geq 1$ we see that
(changing $C$ into $2C$) we obtain the usual Poincar\'e and Cheeger inequalities.
For $p=1$, one recovers the well known fact (see \cite{KLS,bob99}) that Log-concave distributions enjoy Poincar\'e and Cheeger inequalities. Moreover, if we consider the particular case $d\mu(x) = (1/Z_p) \, e^{-|x|^p}$ with $0<p<1$, and
choose $g(x)= e^{|x|^p/2} \, \mathds{1}_{[0,R]}(x)$ for $x\geq 0$ and $g(-x)=-g(x)$, we see
that the weight is optimal in Proposition \ref{ex2}.
\end{remark}

\begin{proof}
The proof follows the same line as the proof of Proposition \ref{exconvexe}, using Lemma \ref{lem:convex2} below.
\end{proof}

\begin{lemma} \label{lem:convex2}
Let $L= \Delta - pV^{p-1} \nabla V \nabla$ for some positive convex function $V$ and $p>0$. Then,
there exists $b, c, R>0$ and $W \geq 1$ such that
$$
LW \leq - \phi(W) + b \mathds{1}_{B(0,R)}
$$
with $\phi(u)= u \, \log^{2(p-1)/p}(c+u)$ increasing. Furthermore, one can choose
$W(x)=e^{\gamma |x|^p}$ for $x$ large.
\end{lemma}

\begin{proof}
We omit the details since we can mimic the proof of Lemma \ref{lem:convex}.
\end{proof}

\begin{remark}
 Changing the values of $b$ and $R$, only the values of $\Phi(u)$ in the large are relevant.
 In other words, one could take $\Phi$ to be an everywhere increasing function which coincides with $u \, \log^{2(p-1)/p}(u)$ for the large $u$'s,
choosing the constants $b$ and $R$ large enough.
\end{remark}

\subsection{Example on the real line}
In this section we give examples on the real line where other techniques can also be done.

Note that in both previous examples we used a Lyapunov function $W=p^{-\gamma}$ for some well chosen
$\gamma>0$. In the next result we give a general statement using such a Lyapunov function in dimension 1.
\begin{proposition}\label{propconc}
Let $d\mu(x)= e^{-V(x)} dx$ be a probability measure on $\R$ for a smooth potential $V$.
We assume for simplicity that $V$ is symmetric. Furthermore, we assume that $V$ is concave on $(R,+\infty)$ for some $R>0$
and that $\left( V''/| V'|^2\right)(x) \, \to \, r>-1/2$ as $x \to \infty$.
Then for some $S>R$ and some $C>0$, it holds
\begin{gather*}
\Var_{\mu}(g) \, \leq \, C \, \int \, |g'(x)|^2 \, \left(1+ \frac{\mathds{1}_{|x|>S}}{|V'|^2(x)}\right)
\, d\mu(x) \, ,\\
\int |g - m| \, d \mu \, \leq \, C \, \int \, |g'(x)| \, \left(1+
\frac{\mathds{1}_{|x|>S}}{|V'|(x)}\right) \, d\mu(x) \,
\end{gather*}
where $m$ is a median of $g$ under $\mu$.
\end{proposition}
\begin{proof}
Since $V'$ is non-increasing on $(R,+\infty)$ it has a limit $l$ at $+\infty$. If $l<0$, $V$ goes
to $- \infty$ at $+\infty$ with a linear rate, contradicting $\int e^{-V} dx < +\infty$. Hence
$l\geq 0$, $V$ is increasing and goes to $+\infty$ at $+\infty$. 

Now choose $W=e^{\gamma V}$ (for large $|x|$). We have
$$
LW= \left(\gamma \, V'' \, - \, (\gamma-\gamma^2)
|V'|^2\right) \, W
$$ so that for $0<\gamma<1$ we have $LW \leq  - \, (\gamma-\gamma^2) |V'|^2W$
at infinity. We may thus choose $\phi(W)= (\gamma-\gamma^2) |V'|^2 \, W$. The corresponding $\phi$
can be built on $(W^{-1}(R),+\infty)$ where $W$ is one to one.
On the other hand,
$$
\phi'(W) \, W'
=
(\gamma-\gamma^2) V' \, W \, \left(2V'' + \gamma |V'|^2\right) \, ,
$$
so that, since $W'>0$, $V'>
0$ and $V''/ |V'|^2>-1/2$ asymptotically, $\phi$ is non-decreasing at infinity for a well chosen $\gamma$.
Then, it is possible to build $\phi$ on a compact interval $[0,a]$ in order to get a smooth increasing function on the whole $\mathbb{R}_+$.

Since $d\mu/dx$ is bounded from above and below on any compact interval, a local Poincar\'e inequality and a local Cheeger
inequality hold on such interval.
hence, it remains to apply Theorem \ref{thmwp} and Theorem \ref{thmcheeg}, since at infinity $\phi'(W)$
behaves like $|V'|^2$.
\end{proof}

\begin{remark}
The example of Proposition \ref{ex2} enters the framework of this proposition, and the general Cauchy
distribution $V(x)= c \log(1+|x|^2)$ does if $c>1$, since $V''/|V'|^2$ behaves asymptotically as
$-1/2c$. Note that the weight we obtain is of good order, applying the inequality with
approximations of $e^{V/2}$.

It is possible to extend the previous proposition to the multi-dimensional setting, but the result
is quite intricate. Assume that $V(x) \to +\infty$ as $|x| \to +\infty$, and that $V$ is concave
(at infinity). The same $W=e^{\gamma V}$ furnishes $LW/W = \gamma \Delta V \, - \, (\gamma-\gamma^2) |\nabla
V|^2$. Hence we may define
$$
\phi(u) = (\gamma-\gamma^2) \, u \, \inf_{A(u)} |\nabla V|^2 \quad
\textrm{ with } \quad A(u)=\{x ; V(x)=\log(u)/\gamma\}
$$
at least for large $u$'s. The main difficulty is to check that $\phi$ is increasing. This could probably be done on specific examples.
\end{remark}

It is known that Hardy-type inequalities are useful tool to deal with functional inequalities of Poincar\'e type in dimension 1
(see \cite{barthe-roberto,bartr03sipm} for recent contributions on the topic).
We shall use now Hardy-type inequalities to relax the hypothesis on $V$ and to obtain the weighted Poincar\'e inequality
of Proposition \ref{propconc}. However no similar method (as far as we know) can be used for the weighted Cheeger inequality,
making the $\phi$-Lyapunov approach very efficient.

\begin{proposition}
 Let $d\mu(x)= e^{-V(x)} dx$ be a probability measure on $\R$ for a smooth potential $V$ that we suppose for simplicity to be even.
Let $\varepsilon \in (0,1)$. Assume that there exists $x_0 \geq 0$ such that $V$ is twice differentiable on $[x_0,\infty)$ and
$$
V'(x) \neq 0, \qquad \frac{|V''(x)|}{V'(x)^2} \leq 1-\varepsilon, \qquad \forall x \geq x_0 .
$$
Then, for some $C>0$, it holds
$$
\Var_{\mu}(g) \, \leq \, C \, \int \, | g'(x)|^2 \, \left(1+ \frac{\mathds{1}_{|x|>{x_0}}}{|V'|^2(x)}\right)
\, \mu(dx) \, .
$$
\end{proposition}

\begin{proof}
Given $\eta$ and using a result of Muckenhoupt \cite{muckenhoupt}, one has for any $G$
$$
\int_0^{+\infty} \, (G(x)-G(0))^2 \, \mu(dx) \, \leq \, 4B \, \int_0^{+\infty} \, G'(x)^2 \,
(1+\eta^2(x)) \, \mu(dx) \, ,
$$
with
$B= \sup_{y>0} \, \left(\int_y^{+\infty} e^{-V(x)} dx \right) \, \left(\int_0^y \, \frac{e^{V(x)}}{1+\eta^2(x)} \, dx\right)$.
Hence, since $V$ is even and $\Var_\mu(g) \leq \int_{-\infty}^0 \, (G(x)-G(0))^2 d\mu  +\int_0^{+\infty} \, (G(x)-G(0))^2 d\mu$, the previous bound applied twice leads to
$$
\Var_\mu(g) \leq 4B \int \, | g'(x)|^2 \, \left(1+\eta(x)^2\right) d\mu .
$$
In particular, one has to prove that
$$
B= \sup_{y>0} \, \left(\int_y^{+\infty} e^{-V(x)} dx \right) \, \left(\int_0^y \, \frac{e^{V(x)}}{1+\frac{\mathds{1}_{|x|>{x_0}}}{|V'|^2(x)}} \,dx\right) < \infty
$$
Consider $y \geq x_0$. Then, (note that $V'>0$ since it cannot change sign and $e^{-V}$ is integrable),
\begin{eqnarray*}
 \int_{x_0}^y \frac{e^{V(x)}}{1+\frac{\mathds{1}_{|x|>{x_0}}}{|V'|^2(x)}} \,dx
& = &
\int_{x_0}^y \frac{V'(x) e^{V(x)} } {V'(x)+\frac{1}{V'(x)}} \,dx
=
\left[ \frac{e^{V}}{V'+\frac{1}{V'}} \right]_{x_0}^y + \int_{x_0}^y e^{V} \frac{V''((V')^2-1)}{((V')^2+1)^2} \\
& \leq &
\frac{e^{V(y)}}{V'(y)+\frac{1}{V'(y)}} + (1-\varepsilon) \int_{x_0}^y e^{V} \frac{(V')^2|(V')^2-1|}{((V')^2+1)^2} \\
& \leq &
\frac{e^{V(y)}}{V'(y)+\frac{1}{V'(y)}} + (1-\varepsilon) \int_{x_0}^y \frac{e^{V}}{1+\frac{1}{(V')^2}}
\end{eqnarray*}
where in the last line we used that $x^2|x^2-1|/(x^2+1)^2 \leq 1/(1+\frac{1}{x^2})$ for $x=V'>0$.
This leads to
$$
 \int_{x_0}^y \frac{e^{V(x)}}{1+\frac{\mathds{1}_{|x|>{x_0}}}{|V'|^2(x)}} \,dx  \leq
\frac{1}{\varepsilon} \frac{e^{V(y)}}{V'(y)+\frac{1}{V'(y)}} .
$$
Similar calculations give (we omit the proof)
$$
\int_y^{+\infty} e^{-V(x)} dx \leq \frac{1}{\varepsilon} \frac{e^{-V(y)}}{V'(y)} \qquad \forall y \geq x_0 .
$$
Combining these bounds and using a compactness argument on $[0,x_0]$, it is not hard to show that $B$ is finite.
\end{proof}

We end this section with distributions in dimension 1 that do not enter the framework of the two previous propositions.
Moreover, the laws we have considered so far are $\kappa$ concave for $\kappa> -\infty$. The last examples shall satisfy $\kappa=-\infty$.

\begin{example}\label{ex3}
Let $q>1$ and define
$$
d \mu(x)= (1/Z_q) \, \left((2+|x|) \, \log^q(2+|x|)\right)^{-1} \, dx \,
=  \, V_q^{-1}(x) dx \, \qquad x \in \mathbb{R} .
$$
The function $V_q$ is convex but $V_q^\gamma$ is no more convex for $\gamma<1$ (hence $\kappa=-\infty$).
We may choose $W(x)=(2+|x|)^2 \, \log^{a}(2+|x|)$ (at least far from 0), which is a $\phi$-Lyapunov function for
$\phi(u)=\log^{a - 1}(2+|u|)$ provided $q > a > 1$ (details are left to the reader). We thus get a weighted inequality
\begin{equation}\label{eqgamma}
\Var_{\mu}(g) \, \leq \, C \, \int \, |\nabla g(x)|^2 \, \left(1 + x^2 \,
\log^2(2+|x|)\right) \, d\mu(x) \, .
\end{equation}
Unfortunately we do not know whether the weight is correct in this situation. The usual choice $g$
behaving like $\sqrt{(2+|x|) \, \log^q(2+|x|)}$ on $(-R,R)$ furnishes a variance
behaving like $R$ but the right hand side behaves like $R \, \log^2 R$.
\medskip

We may even find a Lyapunov functional in the case $V(x)=x \, \log x \, \log^{q}(\log x)$ for
large $x$ and $q>1$, {\it i.e} choose $W(x)=1+|x|^2\log(2+|x|)\log^c\log(2e+|x|)$ with $1<c<q$
for which $\phi(x)$ is merely $\log^{c-1}\log(2e+|x|)$ so that the weight in the Poincar\'e
inequality is $1+|x|^2\log^2(2+|x|)\log^2\log(2e+|x|)$.
 \hfill $\diamondsuit$
\end{example}

%
%


\subsection{Converse inequalities.}\label{secconverse}

This section is dedicated to the study of converse inequalities from $\phi$-Lyapunov function. We start with converse Poincar\'e inequalities
and then we study converse Cheeger inequalities.

\begin{definition}\label{defconvweight}
We say that $\mu$ satisfies a converse weighted Cheeger (resp. Poincar\'e) inequality with
weight $\omega$ if for some $C >0$ and all $g \in \mathcal{A}$
\begin{equation}\label{eqconvwcheeg}
\inf_c \, \int \, {|g-c|} \, \omega \, d\mu \, \leq \, C \, \int \, \sqrt{\Gamma(g)} \,
 d\mu \, ,
\end{equation}
respectively, for all $g \in \mathcal{A}$,
\begin{equation}\label{eqconvwpoinc}
\inf_c \, \int \, |g-c|^2 \, \omega \, d\mu \, \leq \, C \, \int \, \Gamma(g) \, d \mu \, .
\end{equation}
\end{definition}

\subsubsection{Converse Poincar\'e inequalities}
In \cite[Proposition 3.3]{BLweight}, the authors perform a change of function in the weighted Poincar\'e inequality
to get
$$
\inf_c \int (f-c)^2 \, \omega \, d\mu\le\int|\nabla f|^2d\mu.
$$
This method requires that the constant $D$ in the weighted Poincar\'e inequality \eqref{eqwpoinc} (with weight $\eta(x)=(1+|x|)^2$) is not
too big. The same can be done in the general situation, provided the derivative of the weight is
bounded and the constant is not too big.

But instead we can also use a direct approach from $\phi$-Lyapunov functions.

\begin{theorem}[Converse Poincar\'e inequality] \label{thmcwp}
Under the assumptions of Theorem \ref{thmwp}, for any $g \in \mathcal{A}$, it holds
\begin{equation}\label{eqcwp}
\inf_c \, \int (g-c)^2 \, \frac{\phi(W)}{W} \,  d\mu \le (1+b\kappa_U) \, \int \, \Gamma(g) \, d\mu \, .
\end{equation}
\end{theorem}

\begin{proof}
Rewrite the drift condition as
$$
w := \frac{\phi(W)}{W} \le - \frac{LW}{W}+ b \, \mathds{1}_K \, ,
$$
recalling that $W \ge 1$. Set $f=g-c$ with $\int_U (g-c) d\mu = 0$.
Then,
$$
\inf_c \, \int (g-c)^2 \, \frac{\phi(W)}{W} \,  d\mu \leq \int f^2 \, w \, d\mu \le \int- \frac{LW}{W}  \, f^2 \, d\mu\, +\, b \, \int_K \, f^2 \, d\mu \, .
$$
The second term in the right hand side of the latter can be handle using the local Poincar\'e inequality, as in the proof of Theorem \ref{thmwp} (we omit the details). We get $\int_K \, f^2 \, d\mu \leq \kappa_U \int \Gamma(g) d\mu$.
For the first term we use Lemma \ref{lem:philyapounov} with $\psi(x)=x$. This achieves the proof.
\end{proof}

\begin{remark}
In the proof the previous theorem, we used the inequality
$$
\int \, \frac{- LW}{W} \, f^2 \, d\mu \, \leq \, \int \, \Gamma(f) \, d\mu .
$$
By \cite[Lemma 2.12]{CGWW}, it turns out that the latter can be obtained without assuming
that $\Gamma$ is a derivation. In particular the previous Theorem extends to any situation
where $L$ is the generator of a $\mu$-symmetric Markov process (including jump processes) in the
form
$$
\inf_c \, \int (g-c)^2 \, \frac{\phi(W)}{W} \,  d\mu \le (1+b\kappa_U)\, \int \,  - g \, Lg \, d\mu \, .
$$
\hfill $\diamondsuit$
\end{remark}

Now we give two examples to illustrate our result.

\begin{proposition}[Cauchy type law]
Let $d\mu(x) = (V(x))^{-(n + \alpha)} \, dx$ for some positive convex function $V$ and
$\alpha> 0$. Then there exists $C > 0$ such that for all $g$
$$
\inf_c \, \int
(g(x)-c)^2 \, \frac{1}{1+|x|^2} \,  \, d\mu(x) \, \le  \, C \, \int \, |\nabla g|^2 \,
d\mu \, .
$$
\end{proposition}

\begin{proof}
It is a direct consequence of Theorem \ref{thmcwp} and Lemma \ref{lem:convex}.
\end{proof}

\begin{proposition}[Sub exponential like law]
Let $d \mu= (1/Z_p) \, e^{- V^p}$
for some positive convex function $V$ and $p \in (0,1)$.
Then there exists $C > 0$ such that for all $g$
$$
\inf_c \, \int
(g(x)-c)^2 \, \frac{1}{1+|x|^{2(1-p)}} \,  \, d\mu(x) \, \le \, C \,
\int \, |\nabla g|^2 \, d\mu \, .
$$
\end{proposition}

\begin{proof}
Again it is a direct consequence of Theorem \ref{thmcwp} and Lemma \ref{lem:convex2}.
\end{proof}

\subsubsection{Converse Cheeger inequalities}
Here we study the harder converse Cheeger inequalities. The approach by $\phi$-Lyapunov functions works but
some additional assumptions have to be done.

\begin{theorem}[Converse Cheeger inequality]\label{theomconvc}
Under the hypotheses of Theorem \ref{thmcheeg}, assume that $K$ is compact and that either
\begin{itemize}
\item[(1)] \quad $|\Gamma(W,\Gamma(W))| \leq 2 \delta \phi(W) \, (1+\Gamma(W))$ outside $K$, for some $\delta \in (0,1)$
\end{itemize}
or
\begin{itemize}
\item[(2)] \quad $\Gamma(W,\Gamma(W))\geq 0$ outside $K$.
\end{itemize}
Then, there exists a constant $C>0$ such that for any $g \in \mathcal{A}$, it holds
$$
\inf_c \, \int \, |g-c| \, \frac{\phi(W)}{\sqrt{1+\Gamma(W)}} \, d\mu \leq C \, \int
\sqrt{\Gamma (g)} d\mu \, .
$$
\end{theorem}

\begin{remark} \label{rem:dim1}
 Note that using Cauchy-Schwarz inequality,  Assumption $(1)$ is implied by  $\Gamma( \Gamma(W) ) \leq 4 \delta^2 \phi(W)^2 (1+\Gamma(W))$ outside $K$.

On the other hand, in dimension 1 for usual diffusions, we have $\Gamma(W,\Gamma(W))=2 \, |W'|^2 \, W''$. Hence this term is non negative as soon as
$W$ is convex outside $K$.
\end{remark}

\begin{proof}
Let $g \in \mathcal{A}$ and set $f=g-c$ with $c$ satisfying $\int_U (g-c) d\mu =0$.
Recall that $LW \leq - \phi(W) + b \mathds{1}_K$. Hence,
$$
\frac{\phi(W)}{\sqrt{1+\Gamma(W)}} \leq -\frac{LW}{\sqrt{1+\Gamma(W)}} + \frac{b \mathds{1}_K}{\sqrt{1+\Gamma(W)}}
\leq -\frac{LW}{\sqrt{1+\Gamma(W)}} + b \mathds{1}_K .
$$
In  turn,
\begin{eqnarray*}
\int \, |f| \, \frac{\phi(W)}{\sqrt{1+\Gamma(W)}} \, d\mu
& \leq &
- \int \frac{|f|}{\sqrt{1+\Gamma(W)}} \, LW \, d\mu \, + \, b \, \int_K  \, |f| \, d\mu \, .
\end{eqnarray*}
To control the first term we use \eqref{eq:gamma}, the fact that $\Gamma$ is a derivation and Cauchy-Schwarz inequality to get that
\begin{eqnarray*}
- \int \, \frac{|f| }{\sqrt{1+\Gamma(W)}} LW \, d\mu
& = &
\int \, \Gamma\left(\frac{|f|}{\sqrt{1+\Gamma(W)}},W\right) \, d\mu \\
& = &
\int \, \frac{\Gamma(|f|,W)}{\sqrt{1+\Gamma(W)}}  \, d\mu
+
\int \, |f| \Gamma\left(\frac{1}{\sqrt{1+\Gamma(W)}},W\right) \, d\mu \\
& \leq &
\int \, \sqrt{\Gamma(f)} \, d\mu \, - \,  \int |f| \, \frac{\Gamma(W,\Gamma(W))}{2(1+\Gamma(W))^{\frac 32}} \, d\mu \, .
\end{eqnarray*}
Now, we divide the second term of the latter in sum of the integral over $K$ and the integral outside $K$.
Set $M:=\sup_K \frac{|\Gamma(W,\Gamma(W))|}{2(1+\Gamma(W))^{\frac 32}}$.
Under Assumption $(2)$, the integral outside $K$ is non-positive, thus we end up with
$$
\int \, |f| \, \frac{\phi(W)}{\sqrt{1+\Gamma(W)}} \, d\mu \leq \int \, \sqrt{\Gamma(f)} \, d\mu + (M+b) \int_K  \, |f| \, d\mu
$$
while under Assumption $(1)$, we get
$$
\int \, |f| \, \frac{\phi(W)}{\sqrt{1+\Gamma(W)}} \, d\mu \leq \int \, \sqrt{\Gamma(f)} \, d\mu + (M+b) \int_K  \, |f| \, d\mu +
\delta \int \, |f| \, \frac{\phi(W)}{\sqrt{1+\Gamma(W)}} \, d\mu
$$
In any case the term $\int_K  \, |f| \, d\mu$ can be handle using the local Cheeger inequality (we omit the details):
we get $\int_K |f| d\mu \leq \kappa_U \int \sqrt{\Gamma(g)} d\mu$. This ends the proof, since $\Gamma(f)= \Gamma(g)$.
\end{proof}

We apply our result to Cauchy type laws.

\begin{proposition}[Cauchy type laws]\label{propconv2}
Let $d\mu(x)= (V(x))^{-(n+\alpha)} \, dx$ for some positive convex function $V$ and some $\alpha>0$. Then, there exists $C>0$ such that
for any $g$, it holds
$$
\inf_c \, \int \, |g(x)-c| \, \frac{1}{1+|x|} \, \mu(dx) \leq C \, \int
|\nabla g| d\mu \, .
$$
\end{proposition}

\begin{proof}
By Lemma \ref{lem:convex}, we know that $W(x)=|x|^k$ (for $x$ large) is a $\phi$-Lyapunov function for
$\phi(u)= c |u|^{(k-2)/k}$. Note that
$\Gamma(W,\Gamma(W))(x)= (2k-2) k^2 |x|^{3k-4}$ at infinity. Hence Assumption $(2)$ of the previous theorem holds and the theorem applies.
This leads to the expected result.
\end{proof}


The same argument works for sub exponential distributions (we omit the proof).

\begin{proposition}[Sub exponential type laws]\label{propconv3}
Let $d \mu= (1/Z_p) \, e^{- V^p}$
for some positive convex function $V$ and $p \in (0,1)$.
Then there exists $C > 0$ such that for all $g$
$$
\inf_c \, \int \, |g(x)-c| \, \frac{1}{1+|x|^{1-p}} \,
d\mu(x) \leq C \, \int |\nabla g| d \mu \, .
$$
\end{proposition}

%

We end this section with an example in dimension 1.
Consider $d\mu=e^{-V}$ on $\mathbb{R}$, and $W=e^{\gamma V}$ for some $\gamma < 1$. The function $W$ is convex in the large
as soon as $\limsup (|V''|/|V'|^2)<\gamma$ at infinity. Hence we can use remark \ref{rem:dim1} and the previous theorem
to get that, under the hypothesis of Proposition \ref{propconc}, for some $S>0$ and $C>0$
$$
\inf_c \, \int \, |g-c| \, (\mathds{1}_{(-S,S)} + |V'|) \, d\mu \leq C \, \int |g'| \, d\mu  \, .
$$
(we used also that $W$ is a Lyapunov function with $\phi$ satisfying
$\phi(W)=(\gamma-\gamma^2)|V'|^2 W$ (which leads to $\phi(W)/\sqrt{1 + \Gamma(W)}$ of the order of $|V'|$ in the large), see
the proof of Proposition \ref{propconc} for more details).

%

\subsection{Additional comments}

{\it Stability.} As it is easily seen, the weighted Cheeger and Poincar\'e inequalities (and their converse) are stable under
log-bounded transformations of the measure. The Lyapunov approach encompasses a similar property
with compactly supported (regular) perturbations. In fact the Lyapunov aproach is even more robust, let us
illustrate it in the following example: suppose that the measure $\mu=e^{-V}dx$ satisfies a
$\phi$-Lyapunov condition with test function $W$ and suppose that for large $x$, $\nabla V.\nabla
W\ge \nabla V\nabla U$ for some regular (but possibly unbounded) $U$, then there exists $\beta>0$
such that $d\nu=e^{-V+\beta U}dx$ satisfies a $\phi$-Lyapunov condition with the same test
function $W$ and then the same weighted Poincar\'e or Cheeger inequality.
\smallskip

{\it Manifold case.} In fact, many of the results presented here can be extended to the manifold
case, as soon as we can suppose that $V(x)\to\infty$ as soon as the geodesic distance (to some
fixed points) grows to infinity  and of course that a local Poincar\'e inequality or a local
Cheeger inequality is valid. We refer to \cite{CGWW} for a more detailed discussion.


\section{Weak inequalities and isoperimetry.}\label{secweak}

In this section we recall first a result of Bobkov that shows the equivalence between the isoperimetric inequality and
what we have called a weak Cheeger inequality (see \ref{in:wc}).

\begin{lemma}[Bobkov \cite{bob07}]\label{lembob}
Let $\mu$ be a probability measure on $\R^n$. There is an equivalence between the following two
statements (where $I$ is symmetric around $1/2$)
\begin{itemize}
\item[(1)] \quad for all $s>0$ and all smooth $f$ with $\mu$ median equal to 0,
$$
\int \, |f| \, d\mu \, \leq \, \beta(s) \, \int \, |\nabla f| \, d\mu \, + \, s \, \Osc_\mu(f)
\, ,
$$
\item[(2)] \quad for all Borel set $A$ with $0<\mu(A)<1$,
$$
\mu_s(\partial A)\, \geq \, I(\mu(A)) ,
$$
\end{itemize}
where $\beta$ and $I$ are related by the duality relation
$$
\beta(s) \, = \, \sup_{s\leq t \leq \frac 12} \, \frac{t-s}{I(t)} \, ,
\, \quad
I(t) \, = \, \sup_{0<s\leq t} \, \frac{t-s}{\beta(s)} \, \textrm{ for } t \leq \frac 12 \, .
$$
Here as usual $\Osc_\mu(f)={\rm ess\, sup} f - {\rm ess\, inf} f$ and $\mu_s(\partial A)= \liminf_{h \to 0} \, \frac{\mu(0<d(x,A)<h)}{h}$.
\end{lemma}
Recall that in the weak Cheeger inequality, only the values $s \in (0,1/2)$ are relevant since $\int |f| d\mu \leq \frac12\Osc_\mu (f)$. Moreover this Lemma and its proof extend to the general case we are dealing with as soon as the
general coarea formula is satisfied and provided one can approximate indicators by $\sqrt{\Gamma}$
of Lipschitz functions.

\bigskip

Thanks to the previous lemma, we see that isoperimetric results can be derived from weak Cheeger inequalities. We now give two different way to
prove such inequalities. The first one is based on the $\phi$-Lyapunov approach using the converse Cheeger inequalities proved in the previous section. The second one uses instead a transportation of mass technique.

\subsection{From converse Cheeger to weak Cheeger inequalities}

Here we shall first relate converse inequalities to weak inequalities, and then deduce some isoperimetric results on concrete examples.

\begin{theorem}\label{thmweakC}
Let $\mu$ be a probability measure and $\omega$ be a non-negative function satisfying  $\bar{\omega}=\int \omega\, d\mu < +\infty$.
Assume that there exists $C>0$ such that 
$$
\inf_c \, \int |g-c| \, \omega \, d\mu
\le C \, \int \, \sqrt{\Gamma(g)} \, d\mu \, \qquad \forall g \in \mathcal{A} .
$$
Define $F(u)= \mu(\omega < u)$ and
$G(s)=F^{-1}(s):= \inf \{u ; \mu(\omega \leq u) > s\}$.
Then, for all $s>0$ and all  $g \in \mathcal{A}$, it holds
$$
\inf_c \, \int \, |g - c| \, d\mu \, \leq \, \frac{C}{G(s)} \, \int \, \sqrt{\Gamma(g)} \, d\mu \, + \, s \,
\Osc_\mu (f) \, .
$$
\end{theorem}

\begin{proof}
Let $g \in \mathcal{A}$. Define $m_\omega \in \mathbb{R}$ to be a median of $g$ under
$\omega d\mu/\bar{\omega}$. We have
\begin{eqnarray*}
\inf_c \, \int \, |g - c| \, d\mu
& \leq &
\int |g - m_\omega| d\mu \\
& \leq &
\int_{\omega \geq u} |g - m_\omega| \, \frac{\omega}{u} \, d\mu \,
+
\, \int_{\omega<u} \, |g-m_\omega| \, d\mu \\
& \leq &
\frac{1}{u} \, \int \, |g-m_\omega| \, {\omega} \, d\mu \,
+
\, \Osc_\mu (g) \, F(u) \\
& = &
\frac 1u \, \inf_c \, \int \, |g-c| \, \omega \, d\mu \, + \, \Osc_\mu (g)\, F(u) \, .
\end{eqnarray*}
It remains to apply the converse weighted Cheeger inequality and the definition of $G$. Note
that if $F(u)=0$ for $u\leq u_0$ then $G(s)\geq u_0$.
\end{proof}

We illustrate this result on two examples.

\begin{proposition}[Cauchy type laws] \label{prop:convcauchy}
Let $d\mu(x)= V^{-(n+\alpha)}(x) dx$ with $V$ convex and $\alpha>0$. Recall that $\kappa=-1/\alpha$.
Then, there exists a constant $C>0$
such that for any $f$ with $\mu$-median $0$,
$$
\int \, |f| \, d\mu \, \leq \, Cs^\kappa \, \int \, |\nabla f| \, d\mu \, + \,  s \Osc_\mu(f) \qquad \forall s>0 .
$$
Equivalently there exists $C'>0$ such that for any $A \subset \mathbb{R}^n$,
$$
\mu_s(\partial A) \geq C'  \min \left(\mu(A),1-\mu(A) \right)^{1- \kappa} .
$$
\end{proposition}

\begin{proof}
By Proposition \ref{propconv2}, $\mu$ satisfies a converse weighted Cheeger inequality with
weight $\omega(x) = \frac{1}{1+|x|}$. So $F(u)=\mu(\omega < u)=\mu(u^{-1} -1<|x|)$. Since $V$ is
convex, $V(x)\geq \rho |x|$ for large $|x|$ (recall Lemma \ref{lemfranck}), hence using polar
coordinates we have
$$
\mu(|x|>R)
=
\int_{|x|>R} \, V^{-\beta}(x) \, dx \,
\leq \int_{|x|>R} \rho^{-\beta} |x|^{-\beta} dx
\leq
\, c
R^{n-\beta} \, ,
$$
for some $c=c(n,\alpha,\rho)$. The result follows by Theorem \ref{thmweakC}.
The isoperimetric inequality follows at once by Lemma \ref{lembob}.
\end{proof}

\begin{remark}
The previous result recover Corollary 8.4 in \cite{bob07} (up to the constants).
Of course we do not attain the beautiful Theorem 1.2 in \cite{bob07},
where S. Bobkov shows that the constant $C'$ only depends on $\kappa$ and the median of $|x|$.
\end{remark}

\begin{proposition}[Sub exponential type laws] \label{prop:weakexp}
Let $d \mu= (1/Z_p) \, e^{- V^p}$
for some positive convex function $V$ and $p \in (0,1)$.
Then there exists $C > 1$ such that for all $f$ with $\mu$-median $0$,
$$
\int \, |f| \, d\mu \, \leq \, C\log^{\frac 1 p - 1}(C/s) \, \int \, |\nabla f| \, d\mu \, + \,  s \Osc_\mu(f) \qquad \forall s\in (0,1) .
$$
Equivalently there exists $C'>0$ such that for any $A \subset \mathbb{R}^n$,
$$
\mu_s(\partial A) \geq C'  \min \left(\mu(A),1-\mu(A) \right)\log \left(\frac{1}{\min \left(\mu(A),1-\mu(A) \right)}\right)^{1- \frac{1}{p}} .
$$
\end{proposition}
\proof
According to Proposition \ref{propconv3}, $\mu$ verifies the converse Cheeger inequality with the weight function $\omega$ defined by $\omega(x)=1/(1+|x|^{1-p})$ for all $x\in \mathbb{R}^n$.
Moreover, since $V$ is convex, it follows from Lemma \ref{lemfranck} that there is some $\rho>0$ such that $\int e^{\rho |x|^p}\, d\mu(x)<\infty.$ Hence, applying Markov's inequality gives $\mu(|x|> R)\leq Ke^{-\rho R^p}$, for some $K\geq1$. Elementary calculations gives the result.
\endproof

\medskip

\subsection{Weak Cheeger inequality via mass transport}\label{sectransport}
The aim of this section is to study how the isoperimetric inequality, or equivalently the weak Cheeger inequality,
behave under tensor products. More precisely, we shall start with a probability measure $\mu$ on the real line $\R$ and derive weak Cheeger inequalities for $\mu^n$ with explicit constants.

We need some notations. For any probability measure $\mu$ (on $\R$) we denote by $F_\mu$ the cumulative distribution function of $\mu$ which is defined by \[F_\mu(x)=\mu(-\infty,x],\quad \forall x \in \R.\] It will be also
convenient to consider the tail distribution function $\bF$ defined by
\[\bF(x)=1-F_\mu(x)=\mu(x,+\infty),\quad \forall x\in \R.\] The isoperimetric function of $\mu$ is defined by
\begin{equation}\label{eq:isopfonc}
J_\mu = F_\mu' \circ F_\mu^{-1} \, .
\end{equation}

In all the sequel, the two sided exponential measure $d\nu(x) = \frac{1}{2}e^{-|x|}\,dx$, $x \in \R$ will play the role of a reference probability measure. We will set $F_\nu = F$ and $J_\nu=J$ for simplicity. Note that the
isoperimetric function $J$ can be explicitly computed: $J(t)=\min (t, 1-t)$, $t \in [0,1]$.
\smallskip

\subsubsection{A general result}We are going to derive a weak Cheeger inequality starting from a well known Cheeger inequality for $\nu^n$
obtained in \cite{BH97} and using a transportation idea developed in \cite{gozjfa}. Our
result will be available for a special class of probability measures on $\R$ which is described in
the following lemma.
\begin{lemma}\label{lem:convexity}
Let $\mu$ be a symmetric probability measure on $\R$ ; the following propositions are equivalent
\begin{enumerate}
\item The function $\log\bF$ is convex on $\R^+$, \item The function $J/J_\mu$ is non increasing
on $(0,1/2]$ and non decreasing on $[1/2,1)$.
\end{enumerate}
Furthermore, if $d\mu(x)=e^{-\Phi(|x|)}\,dx$ with $\Phi:\R^+\to\R$ concave, then $\log \bF$ is
convex on $\R^+.$
\end{lemma}
\proof The equivalence between (1) and (2) is easy to check. Now suppose that $\mu$ is of the form
$d\mu(x)=e^{-\Phi(|x|)}\,dx$ with a concave $\Phi$. Then for $r \in \R^+$,
\[
(\log \bF)''(r) = \frac{e^{-\Phi(r)}}{ \left( \int^{\infty}_r e^{-\Phi(s)}\,ds \right)^2} \left(
\Phi'(r) \int^{\infty}_r e^{-\Phi(s)}\,ds - e^{-\Phi(r)} \right)
\]
where $\Phi'$ is the right derivative. Since  $\Phi$ is concave, $\Phi'$ is non-increasing. It
follows that
\[
\Phi'(r) \int_r^\infty e^{-\Phi(s)}\,ds \geq \int_r^\infty \Phi'(s) e^{-\Phi(s)}\,ds =
e^{-\Phi(r)} .
\]
The result follows.
\endproof

Recall that distributions satisfying (1) in the previous lemma are known as ``Decreasing Hazard
Rate'' distributions. We refer to \cite{BMP} for some very interesting properties of these
distributions (unfortunately less powerful than the Increasing Hazard Rate situation).
\smallskip

Using a mass transportation technique, we are now able to derive a weak Cheeger inequality for
product measures on $\R^n$. Dimension dependence is explicit, as well as the constants.

\begin{theorem} \label{th:weakcheeger}
Let $\mu$ be a symmetric probability measure on $\R$ absolutely continuous with respect to the
Lebesgue measure. Assume that $\log\overline{F}_\mu$ is convex on $\R^+$.

Then, for any $n$, any bounded smooth function $f: \R^n \to \R$ satisfies
\begin{equation} \label{eq:weakcheeger}
\int |f - m|\, d\mu^n \leq \kappa_1\frac{s}{J_\mu(s)}\int |\nabla f|\, d\mu^n + \kappa_2ns \Osc
(f), \quad \forall s\in(0,1/2),
\end{equation}
where $m$ is a median of $f$ under $\mu^n$, $\kappa_1=2\sqrt{6}$ and $\kappa_2=2(1+2\sqrt{6})$.
\end{theorem}

\begin{remark} \label{rem:s}
Note that $\int |f - m| d\mu^n \leq \Osc (f)$. Hence only the values $s \leq (\kappa_2n)^{-1}$ are
of interest in \eqref{eq:weakcheeger}. \hfill $\diamondsuit$
\end{remark}

\begin{proof}
Recall that $\nu$ is the two sided exponential distribution. Fix the dimension $n$ and $r>0$. By
\cite[Inequality (6.9)]{BH97}, any locally Lipschitz function $h : \R^n \to \R$ with $\int
|h|\,d\nu^n < \infty$ satisfies
\begin{equation} \label{eq:bh}
\int |h-m_{\nu^n}(h)|\, d\nu^n \leq \kappa_1 \int |\nabla h|\, d\nu^n
\end{equation}
where $m_{\nu^n}(h)$ is a median of $h$ for $\nu^n$ and $|\cdot|$ is the Euclidean norm on $\R^n$.

Consider the  map $T^n : \R^n \to \R^n$, that pushes forward $\nu^n$ onto $\mu^n$,
defined by $(x_1,\ldots,x_n) \mapsto (T(x_1), \ldots, T(x_n))$ with $T = F_\mu^{-1} \circ F$. By
construction, any $f : \R^n \to \R$ satisfies $\int f(T^n)\, d\nu^n = \int f\, d\mu^n$.

Next, for $t \geq 0$ let $B(t) = \{ x=(x_1,\ldots,x_n) : \max_i |x_i| \leq t\}$. Fix $a > 0$ that
will be chosen later and consider $g:\R \to [0,1]$ defined by $g(x)= \left(1 - \frac{1}{a}(x-r)_+
\right)_+$ with $X_+ = \max(X,0)$. Set $\varphi(x)=g(\max_i(|x_i|))$, $x \in \R^n$. The function $\varphi$ is
locally Lipschitz.

Finally let $f : \R^n \to \R$ be smooth and bounded. We assume first that $0$ is a $\mu^n$-median
of $f$. Furthermore, by homogeneity of \eqref{eq:weakcheeger} we may assume that $\Osc (f) =1$ in
such a way that $\|f\|_\infty \leq 1$. It follows from the definition of the median that
\begin{align*}
\int |f|\, d\mu^n & \leq
\int |f - m_{\nu^n}((f\varphi)(T^n))|\, d\mu^n \\
& \leq
\int |f  \varphi - m_{\nu^n}((f\varphi)(T^n)) |\, d\mu^n + \int |f (1 - \varphi)|\, d\mu^n \\
& \leq \int |f\varphi - m_{\nu^n}((f\varphi)(T^n))|\, d\mu^n + \mu^n \left( B(r)^c \right).
\end{align*}

Note that the assumption on $\log \bF$ guarantees that $T'\circ T^{-1}$ is non-decreasing on
$\R^+$. Hence, using \eqref{eq:bh}, the triangle inequality in $\ell^2(\R^n)$, the fact that $0
\leq \varphi \leq 1$ on $\R^n$ and $\varphi = \partial_i \varphi=0$ on $B(r+a)^c$ imply  that
\begin{align*}
& \int|f\varphi - m_{\nu^n}((f\varphi)(T^n))|\, d\mu^n
 =
\int |(f \varphi)(T^n) - m_{\nu^n}((f\varphi)(T^n))|\, d\nu^n \\
& \leq
\kappa_1 \int \sqrt{\sum_{i=1}^n T'(x_i)^2 \left((\varphi \partial_i f)(T^n)  +
(f \partial_i \varphi)(T^n)\right)^2}\, d\nu^n \\
&   =
\kappa_1 \int \sqrt{\sum_{i=1}^n T'\circ T^{-1}(x_i)^2 \left(\varphi \partial_i f +
f \partial_i \varphi \right)^2}\, d\mu^n \\
&  \leq \kappa_1 \int \sqrt{\sum_{i=1}^n T'\circ T^{-1}(x_i)^2 \left(\varphi \partial_i
f\right)^2}\, d\mu^n
+  \kappa_1 \int \sqrt{\sum_{i=1}^n T'\circ T^{-1}(x_i)^2 \left(f \partial_i \varphi \right)^2}\, d\mu^n \\
&  \leq \kappa_1 T'\circ T^{-1}(r+a) \left( \int |\nabla f|\, d\mu^n + \int |\nabla \varphi|\,
d\mu^n \right).
\end{align*}
Note that $|\nabla \varphi| \leq 1/h$ on $B(r+a) \setminus B(r)$ and $|\nabla \varphi|=0$
elsewhere $\mu^n$-almost surely. Hence,
\begin{multline}\label{eq:weakcheeger1}
\int|f|\,d\mu^n
 \leq
\kappa_1T'\circ T^{-1}(r+a) \left( \int |\nabla f| d\mu^n + \frac{1}{a} \mu^n\left( B(r+a)
\setminus B(r) \right) \right)+\mu^n \left( B(r)^c \right).
\end{multline}
Since $\mu$ is symmetric, we have
$$
G(t):=\mu^n\left(B(t)\right) = \left( 1 - 2\overline{F}_\mu(t) \right)^n .
$$
Hence,
\begin{align*}
\lim_{a \to 0} \frac{1}{a} \mu^n\left( B(r+a) \setminus B(r) \right) & = G'(r) =
2n F_\mu'(r) \left( 1 - 2\overline{F}_\mu(t) \right)^{n-1} \\
& \leq 2n F_\mu'(r) .
\end{align*}
On the other hand, since the function $x\mapsto 1-(1-2x)^n$ is concave on $[0,1/2]$, one has:
$1-(1-2x)^n\leq 2nx$ for all $x\in [0,1/2]$. As a consequence,
\[\mu^n(B(r)^c) = 1 - G(r) = 1 - (1-2\overline{F}_\mu(r))^n \leq 2n\overline{F}_\mu(r),\]
for all $r\geq 0.$

Letting $a$ go to $0$ in \eqref{eq:weakcheeger1} leads to
$$
\int|f|\, d\mu^n \leq \kappa_1T'\circ T^{-1}(r) \int |\nabla f|\, d\mu^n + 2n\kappa_1 T'\circ
T^{-1}(r) F_\mu'(r) + 2n\overline{F}_\mu(r).
$$
Note that $T' \circ T^{-1} =  J \circ F_\mu/F_\mu' = \min(F_\mu,1-F_\mu)/F_\mu'$. Hence, for $r
\geq 0$, $$T' \circ T^{-1} (r) F_\mu'(r) = \frac{1 - F_\mu(r)}{F_\mu'(r)} F_\mu'(r) =
\overline{F}_\mu(r) \, .$$ It follows that
$$
\int|f|\, d\mu^n \leq \kappa_1\frac{\overline{F}_\mu(r)}{F_\mu'(r)}\int |\nabla f|\, d\mu^n +
n\kappa_2 \overline{F}_\mu(r),
$$
for all $r\geq 0$. Using the symmetry of $\mu$ it is easy to see that
$F_\mu'\circ\overline{F}_\mu^{-1}(t)=J_\mu(t)$ for all $t\in (0,1/2).$ Consequently, one has
$$
\int|f|\, d\mu^n \leq \kappa_1\frac{s}{J_\mu(s)}\int |\nabla f|\, d\mu^n + \kappa_2ns,
$$
for all $s\in(0,1/2)$. For general $f: \R^n \to \R$ with $\mu^n$-median $m$, we apply the result
to $f-m$. This ends the proof.
\end{proof}

Combining this theorem with Bobkov's Lemma \ref{lembob} we immediately deduce

\begin{corollary}\label{cor:isotransport}
Let $\mu$ be a symmetric probability measure on $\R$ absolutely continuous with respect to the
Lebesgue measure. Assume that $\log \bF$ is convex on $\R^+$. Then, for any $n$, any Borel set $A
\subset \R^n$ satisfies
\begin{equation} \label{eq:iso}
(\mu^n)_s(\partial A) \geq   \frac{n\kappa_2}{\kappa_1} J_\mu \left( \frac{\min(\mu^n
(A),1-\mu^n(A))}{2n\kappa_2} \right).
\end{equation}
\end{corollary}
\begin{proof}
According to Lemma \ref{lembob}, if $\mu(A)\leq 1/2$ (the other case is symmetric),
$(\mu^n)_s(\partial A) \geq I(\mu^n(A))$ with $I(t)=\sup_{0<s\leq t} \frac{t-s}{\beta(s)}$, for
$t\leq 1/2$, where according to the previous theorem
$$\beta(s) = \frac{\kappa_1}{n \, \kappa_2} \, \frac{s}{J_\mu(s/n \kappa_2)} \, , $$ for $
 s \leq n\kappa_2/2$ hence for $s\leq 1/2$.  This yields
$$I(t) = \sup_{0<s\leq t} \, \frac{t-s}{\kappa_1} \, \frac{J_\mu(s/n \kappa_2)}{(s/n \kappa_2)} \,
.$$ In order to estimate $I$ we use the following: first a lower bound is obtained
for $s=t/2$ yielding the statement of the corollary. But next according to Lemma
\ref{lem:convexity}, the slope function $J_\mu(v)/v$ is non-decreasing, so that $$I(t) \, \leq \,
\sup_{0<s\leq t} \, \frac{t-s}{\kappa_1} \, \frac{J_\mu(t/n \kappa_2)}{(t/n \kappa_2)} \, \leq \,
\frac{n \kappa_2}{\kappa_1} \, J_\mu(t/n \kappa_2) \, .$$ Remark that we have shown that for
$t\leq 1/2$
\begin{equation}\label{eq:bonneI}
\frac{n \kappa_2}{\kappa_1} \, J_\mu(t/2n \kappa_2) \, \leq \, I(t) \, \leq \, \frac{n
\kappa_2}{\kappa_1} \, J_\mu(t/n \kappa_2) \, ,
\end{equation}
so that up to a factor 2 our estimate is of good order.
\end{proof}

\bigskip

\subsubsection{Application: Isoperimetric profile for product measures with heavy
tails}\label{secisoptransport}

Here we apply the previous results to product of the measures
\begin{equation}\label{eq:muphi}
\mu(dx)=\mu_\Phi(dx)= Z_\Phi^{-1}\exp\{-\Phi(|x|)\} dx \, ,
\end{equation}
$x \in \R$, with $\Phi$ concave.

For even measures on $\R$ with positive density on a segment, Bobkov and Houdr\'e \cite[Corollary
13.10]{bobkh97cbis} proved that solutions to the isoperimetric problem can be found among
half-lines, symmetric segments and their complements. More precisely, one has for $t\in(0,1)$
\begin{equation} \label{eq:imu}
I_\mu(t)=\min\left(J_\mu(t),2J_\mu\Big(\frac{\min(t,1-t)}{2} \Big) \right).
\end{equation}

Under few assumptions on $\Phi$, $I_\mu$ compares to the function
$$
L_\Phi(t)=\min(t,1-t) \Phi' \circ \Phi^{-1} \left(\log
   \frac{1}{\min(t,1-t)} \right),
$$
where $\Phi'$ denotes the right derivative. More precisely,
\begin{proposition} \label{prop:is}
   Let $\Phi : \R^+ \rightarrow \R$ be a non-decreasing concave function
   satisfying $\Phi(x)/x \to 0$ as $x \to \infty$.
   Assume that in a neighborhood of $+\infty$ the function $\Phi$ is ${\mathcal C}^2$
   and there exists $\theta >1$ such that $\Phi^\theta$ is convex. Let $\mu_\Phi$ be defined in \eqref{eq:muphi}. Define $F_\mu$ and $J_\mu$ as in \eqref{eq:isopfonc}.

Then,
\begin{equation*}
\lim_{t \rightarrow 0} \frac{J_\mu(t)}{t \Phi' \circ \Phi^{-1}(\log \frac 1t)} = 1 .
\end{equation*}
Consequently, if $\Phi(0)<\log 2$, $L_\Phi$ is defined on  $[0,1]$ and there exist constants $k_1, k_2>0$ such
that for all $t \in [0,1]$,
$$
k_1 L_\Phi(t) \le J_\mu(t) \le k_2 L_\Phi(t) .
$$
\end{proposition}

\begin{remark}
This result appears in \cite{bart01lcbe,bobkov-zeg} in the particular case $\Phi(x)=|x|^p$ and in
\cite{BCR3} for $\Phi$ convex and $\sqrt \Phi$ concave. \hfill $\diamondsuit$
\end{remark}

The previous results together with Corollary \ref{cor:isotransport} lead to the following
(dimensional) isoperimetric inequality.

\begin{corollary} \label{cor:iso2}
Let $\Phi : \R^+ \rightarrow \R$ be a non-decreasing concave function
   satisfying $\Phi(x)/x \to 0$ as $x \to \infty$ and $\Phi(0) < \log 2$.
   Assume that in a neighborhood of $+\infty$ the function $\Phi$ is ${\mathcal C}^2$
   and there exists $\theta >1$ such that $\Phi^\theta$ is convex.
   Let $d \mu(x) = Z_\Phi^{-1}e^{- \Phi(|x|)}dx$ be a probability
   measure on $\R$. Then, $$
I_{\mu^n}(t) \geq c \min(t,1-t) \Phi' \circ \Phi^{-1}\left(\log \frac{n}{\min(t,1-t)}\right)
\qquad \forall t \in [0,1], \; \forall n
$$
for some constant $c>0$ independent on $n$.
\end{corollary}

\begin{remark}
Note that there is a gain of a square root with respect to the results in \cite{BCR2}. \hfill
$\diamondsuit$
\end{remark}

For the clarity of the exposition, the rather technical proofs of Proposition \ref{prop:is} and Corollary \ref{cor:iso2} are postponed to the Appendix.

We end this section with two examples.

\begin{proposition}[Sub-exponential law]
Consider the probability measure $\mu$ on $\R$, with density
$Z_p^{-1}e^{-|x|^p}$, $p \in (0,1]$. There is a constant $c$ depending only on $p$ such that for all $n\geq 1$ and all $A\subset \R^n$,
\[\mu^n_s(\partial A)\geq c \min(\mu^n(A),1-\mu^n(A))\log\left(\frac{n}{\min(\mu^n(A),1-\mu^n(A))}\right)^{1-\frac{1}{\beta}} .\]
\end{proposition}
\proof
The proof follows immediately from Corollary \ref{cor:iso2}.
\endproof
\begin{remark}
Let $I_{\mu^n}(t)$ be the isoperimetric profile of $\mu^n$.
The preceding bound combined with the upper bound of \cite[Inequality (4.10)]{BCR2} gives
$$
c(p)\, t\left(\log\Big(\frac{n}{t}\Big) \right)^{1-1/p} \le I_{\mu^{n}}(t) \leq
c'(p) t\log(1/t) \left(\log\Big(\frac{n}{\log(1/t)}\Big) \right)^{1-1/p}
$$
for any $n\ge\log(1/t)/\log2$ and $t\in (0,1/2)$. Hence, we obtain the right logarithmic behavior
of the isoperimetric profile in term of the dimension $n$. This result extends the corresponding
one obtained in section \ref{secweak} for this class of examples.\hfill $\diamondsuit$
\end{remark}

\medskip

More generally consider the probability measure $\mu=Z^{-1}e^{-|x|^p \log( \gamma + |x|)^\alpha}$, $p \in (0,1]$, $\alpha \in \R$ and
$\gamma= \exp\{2|\alpha|/(p(1-p))\}$ chosen in such a way that $\Phi(x)=|x|^p \log(
\gamma + |x|)^\alpha$ is concave on $\R^+$. The assumptions of Corollary \ref{cor:iso2} are
satisfied. Hence, we get that
$$
I_{\mu^n}(t) \geq c(p,\alpha) t\left(\log\Big(\frac{n}{t}\Big) \right)^{1-1/p} \left(
\log  \log \left(e + \frac{n}{t} \right) \right)^\frac{\alpha}{p}, \qquad t\in (0,1/2).
$$

\bigskip

Cauchy laws do not enter the framework of Corollary \ref{cor:iso2}. Nevertheless, explicit
computations can be done.

\begin{proposition}[Cauchy distributions]
Consider $d\mu(x)= \frac{\alpha}{2(1+|x|)^{1+\alpha}}\,dx$ on $\R$, with $\alpha>0$. There is $c>0$ depending only on $\alpha$ such that for all $n\geq 1$ and all $A\subset \R^n$,
\[\mu^n_s(\partial A)\geq c\frac{\min(\mu^n(A),1-\mu^n(A))^{1 +
\frac{1}{\alpha}}}{n^\frac{1}{\alpha}} .\]
\end{proposition}
\proof
Since $1-F_{\mu}(r) = \frac{1}{2(1+r)^\alpha}$ for $r \in \R^+$, $\log(1-F_{\mu})$
is convex on $\R^+$. Moreover $J_{\mu}(t)=\alpha 2^{1/\alpha}\min(t,1-t)^{1+1/\alpha}$, and so the result follows by Corollary \ref{cor:isotransport}.
\endproof
\begin{remark}

Note that, since $J_{\mu}(t)=\alpha 2^{1/\alpha}\min(t,1-t)^{1+1/\alpha}$, one has
$$
I_{\mu}(t)=\alpha t^{1+1/\alpha},\qquad \forall t\in(0,1/2).
$$
Hence, our results reads as
$$
I_{\mu^n}(t) \geq c\, \frac{t}{n^{1/\alpha}} t^{1/\alpha}
$$
for some constant $c$ depending only on $\alpha$. Together with \cite[Inequality (4.9)]{BCR2}
(for the upper bound) our results gives for any $n\geq\log(1/t)/\log2$ and $t\in (0,1/2)$
$$
c\, \frac{t}{n^{1/\alpha}} t^{1/\alpha} \le I_{\mu^n}(t) \leq c'\,
\frac{t}{n^{1/\alpha}} \log(1/t)^{1+1/\alpha}.
$$
Again, we get the correct polynomial behavior in the dimension $n$. \hfill $\diamondsuit$
\end{remark}

\bigskip


\section{Weighted Poincar\'e inequalities for some spherically symmetric probability measures with heavy tails}\label{secspherique}

In this section we deal with spherically symmetric probability measures $d\mu(x) = h(|x|)dx$ on
$\mathbb{R}^n$ with $|\cdot|$ the Euclidean distance. In polar coordinates, the measure $\mu$ with density $h$ can be viewed as the distribution of $\xi \theta$, where $\theta$ is a random vector uniformly
distributed on the unit sphere $S^{n-1}$, and $\xi$ (the radial part) is a random variable
independent of $\theta$ with distribution function
\begin{equation}\label{eq:polar1}
\mu\left\{|x| \leq r \} \right) = n \omega_n \int_0^r s^{n-1} h(s) ds \, ,
\end{equation}
where $\omega_n$ denotes the volume of the unit ball in $\mathbb{R}^n$. We shall denote by
$\rho_\mu(r)=n \omega_n r^{n-1} h(r)$ the density of the distribution of $\xi$, defined on
$\mathbb{R}_+$.

Our aim is to obtain weighted Poincar\'e inequalities with explicit constants for $\mu$ on $\R^n$ of the forms $d\mu(x)=\frac{1}{Z}\frac{1}{\left(1+|x|\right)^{(n+\alpha)}}\,dx$ with $\alpha>0$ or
$d\mu(x)=\frac{1}{Z}e^{-|x|^p}\,dx$, with $p\in (0,1)$. To do so we will apply a general radial transportation technique which is explained in the following result.

Given an application $T:\R^n\to\R^n$, the image of $\mu$ under $T$ is by definition the unique probability measure $\nu$ such that \[\int f\,d\nu=\int f\circ T\,d\mu,\qquad\forall f.\]
In the sequel, we shall use the notation $T\sharp\mu$ to denote this probability measure.

\begin{theorem}[Transportation method] \label{thm:transp-sph}
Let $\mu$ and $\nu$ be two spherically symmetric probability measures on $\R^n$ and suppose that $\mu=T\sharp\nu$
with $T$ a radial transformation of the form: $T(x)=\varphi(|x|)\frac{x}{|x|}$, with $\varphi:\R^+\to\R^+$ an increasing function with $\varphi(0)=0$.

If $\nu$ satisfies Poincar\'e inequality with constant $C$, then $\mu$ verifies the following weighted Poincar\'e inequality
\[\Var_\mu(f)\leq C \int \omega(|x|)^2 |\nabla f|^2\,d\mu(x),\qquad \forall f,\]
with the weight $\omega$ defined by
\[\omega(r)=\max\left( \varphi'\circ\varphi^{-1}(r),\frac{r}{\varphi(r)}\right).\]

If one suppose that $\nu$ verifies Cheeger inequality with constant $C$, then $\mu$ verifies the following weighted Cheeger inequality
\[\int |f-m|\,d\mu\leq C\int \omega(|x|)|\nabla f|(x)\,d\mu(x),\qquad \forall f,\]
with the same weight $\omega$ as above and $m$ being a median of $f$.

Finally, if the function $\varphi$ is convex, then $\omega(r)=\varphi'\circ\varphi^{-1}(r).$
\end{theorem}

\begin{remark}
In \cite{Wang08}, Wang has used a similar technique to get weighted logarithmic Sobolev inequalities.
\end{remark}

\proof
Consider a locally Lipschitz function $f : \R^n \to \R$ ; it follows from the minimizing property of the variance and the Poincar\'e inequality verified by $\nu$ that
\begin{align*}
\Var_\mu(f) \leq \int \left(f - \int f\,d\nu\right)^2\, d\mu
 =
\int \left(f(T) - \int fd\nu \right)^2\, d\nu \leq C \int |\nabla (f \circ T) |^2\, d\nu.
\end{align*}
In polar coordinates we have
\begin{eqnarray*}
 |\nabla (f \circ T) |^2
& = & \left[ \frac{\partial }{\partial r} (f \circ T)\right]^2 + \frac{1}{r^2} \left|
\nabla_\theta (f\circ T)\right|^2
 =
\left(\frac{\partial f}{\partial r} \right)^2 \!\!\circ T  \times {\varphi'}^2
+ \frac{1}{r^2} \left| \nabla_\theta f\right|^2 \\
& = & \left(\frac{\partial f}{\partial r}  \right)^2 \!\! \circ T \times \left(\varphi'\circ \varphi^{-1}
\circ \varphi \right)^2 + \frac{1}{(\varphi^{-1} \circ \varphi)^2} \left| \nabla_\theta f\right|^2 \circ T .
\end{eqnarray*}

Moreover, denoting by $d\theta$ the normalized Lebesgue measure on $S^{n-1}$, and using the notations introduced in the beginning of the section, the previous inequality reads
\begin{eqnarray*}
\Var_\mu(f) & \leq & C \iint  \left(\left(\frac{\partial f}{\partial r}  \right)^2 \!\! \circ T \times
\left(\varphi'\circ \varphi^{-1} \circ \varphi \right)^2 + \frac{1}{(\varphi^{-1} \circ \varphi)^2} \left| \nabla_\theta
f\right|^2 \circ T\right)
\rho_\nu(r) dr d\theta  \\
& = & C \iint \left(  \left(\frac{\partial f}{\partial r}\right)^2 \times \left( \varphi'\circ
\varphi^{-1} \right)^2 + \frac{1}{(\varphi^{-1})^2} \left| \nabla_\theta f\right|^2 \right)
\rho_\mu(r) dr d\theta  \\
& \leq & C \iint \omega^2(r) \left(
\left(\frac{\partial f}{\partial r}\right)^2 +  \frac{1}{r^2}\left| \nabla_\theta f\right|^2
\right)
\rho_\mu(r) dr d\theta  \\
& =& C \int \omega^2 (|x|) |\nabla f|^2 d\mu
\end{eqnarray*}
where we used the fact that the map $\varphi$ transports $\rho_\nu\,dr$ onto $\rho_\mu\,dr$. The proof of the Cheeger case follows exactly in the same way.

Now, let us suppose that $\varphi$ is convex. Since $\varphi$ is convex and $\varphi(0)=0$, one has $\frac{\varphi(r)}{r} \leq \varphi'(r)$. This implies at once that $\omega(r)=\varphi'\circ\varphi^{-1}$ and achieves the proof.
\smallskip
\endproof
To apply Theorem \ref{thm:transp-sph}, one needs a criterion for Poincar\'e inequality.
The following theorem is a slight adaptation of a result by Bobkov \cite[Theorem 1]{bobsphere}.
\begin{theorem}\label{bobsphere}
Let $d\nu(x)=h(|x|)\,dx$ be a spherically symmetric probability measure on $\R^n.$ Define as before $\rho_\nu$ as the density of the law of $|X|$ where $X$ is distributed according to $\nu$ and suppose that $\rho_\nu$ is a log-concave
function. Then $\nu$ verifies the following Poincar\'e inequality
\[\Var_\nu(f)\leq C_\nu \int |\nabla f|^2\,d\nu, \qquad \forall f\]
with $C_\nu=12\left(\int r^2\rho_\nu(r)\,dr -\left(\int r\rho_\nu(r)\,dr\right)^2\right)+\frac{1}{n}\int r^2\rho_\nu(r)\,dr.$
\end{theorem}
\proof
We refer to \cite{bobsphere}.
\endproof

\begin{proposition}[Generalized Cauchy distributions]
The probability measure $d\mu(x)=\frac{1}{Z}\frac{1}{(1+|x|)^{(n+\alpha)}}$ on $\R^n$ with $\alpha>0$ verifies the weighted Poincar\'e inequality
\[\Var_\mu(f)\leq C_\textrm{opt}\int \left(1+|x|\right)^2|\nabla f|^2\,d\mu(x),\qquad \forall f.\]
where the optimal constant $C_\textrm{opt}$ is such that
\[\sum_{k=0}^{n-1} \frac{1}{(\alpha+k)^2} \leq C_\textrm{opt} \leq 14 \sum_{k=0}^{n-1} \frac{1}{(\alpha+k)^2} 
.\]
\end{proposition}

\begin{remark}
Note that, comparing to integrals, we have
$$
\frac{1}{\alpha^2} + \frac{n-1}{(\alpha+1)(\alpha + n)}
\leq
\sum_{k=0}^{n-1} \frac{1}{(\alpha+k)^2}
\leq
\frac{1}{\alpha^2} + \frac{n -1}{\alpha (\alpha+ n-1)} .
$$
Since $\alpha ^2 \sum_{k=0}^{n-1} \frac{1}{(\alpha+k)^2} \to n$ when $\alpha \to \infty$, applying the previous weighted Poincar\'e inequality
 to $g(\alpha x)$, making a change of variables, and letting $\alpha$ tend to infinity lead to
$$
\Var_\nu(f) \leq 14 n \int |\nabla f|^2\,d\nu
$$
with $d\nu(x)=(1/Z)e^{-|x|} dx$. Moreover, the optimal constant in the latter is certainly greater than $n$. This recover
(with 14 instead of 13) one particular result of Bobkov \cite{bobsphere}.
\end{remark}

\proof
Define $\psi(r)=\ln(1+r)$, $r>0$ and let $\nu$ be the image of $\mu$ under the radial map $S(x)=\psi(|x|)\frac{x}{|x|}$. Conversely, one has evidently that $\mu$ is the image of $\nu$ under the radial map $T(x)=\varphi(|x|)\frac{x}{|x|}$, with $\varphi(r)=\psi^{-1}(r)=e^r-1$ (which is convex). To apply Theorem \ref{thm:transp-sph}, one has to check that $\nu$ verifies Poincar\'e inequality. 

Elementary computations yield
\[\frac{d\nu}{dx}(x)=\frac{1}{Z} \left(\frac{e^{|x|}-1}{|x|}\right)^{n-1}e^{(1-n-\alpha)|x|}
\qquad\text{and}\qquad
\rho_\nu(r)=\frac{n\omega_n}{Z}\left(1-e^{-r}\right)^{n-1}e^{-\alpha r}\]
It is clear that $\log \rho_\nu$ is concave. So we may apply Theorem \ref{bobsphere} and conclude that $\nu$ verifies Poincar\'e inequality with the constant $C_\nu$ defined above.

Define \[H(\alpha)=\int_0^{+\infty} e^{- \alpha r}(1-e^{-r})^{n-1}\,dr=\int_0^1 u^{\alpha-1}(1-u)^{n-1}\,du.\]
Then $\int r\rho_\nu(r)\,dr=-\frac{H'(\alpha)}{H(\alpha)}$ and $\int r^2\rho_\nu(r)\,dr=\frac{H''(\alpha)}{H(\alpha)}$.
Integrations by parts yield
\[H(\alpha)=\frac{(n-1)!}{(\alpha + n -1)(\alpha + n -2)\cdots (\alpha)}.\]
So,
\[H'(\alpha)=-H(\alpha)\sum_{k=0}^{n-1} \frac{1}{\alpha + k}
\quad\text{and}\quad
 H''(\alpha)=H(\alpha)\left[\left(\sum_{k=0}^{n-1} \frac{1}{\alpha+k}\right)^2+\sum_{k=0}^{n-1} \frac{1}{(\alpha + k)^2}\right].\]
This gives, using Cauchy-Schwarz inequality 
\[C_\nu=13\sum_{k=0}^{n-1}\frac{1}{(\alpha + k)^2} + \frac{1}{n}\left(\sum_{k=0}^{n-1} \frac{1}{\alpha + k}\right)^2
\leq
14\sum_{k=0}^{n-1}\frac{1}{(\alpha + k)^2}
.\]

Now, suppose that there is some constant $C$ such that the inequality $\Var_\mu(f)\leq C \int (1+|x|)^2|\nabla f|^2\,d\mu$ holds for all $f$. We want to prove that $C\geq \sum_{k=0}^{n-1} \frac{1}{(\alpha+k)^2}$. To do so let us test this inequality on the functions $f_a(x)=\frac{1}{(1+|x|)^a}$, $a>0$. Defining $F(r)=\int \frac{1}{(1+|x|)^{n+r}}\,dr$, for all $r>0$, one obtains immediately
\[C\geq \frac{1}{a^2}\frac{F(2a+\alpha)F(\alpha)-F(a+\alpha)^2}{F(\alpha)F(2a+\alpha)}.\]
But a Taylor expansion easily shows that the right hand side goes to $K=\frac{F''(\alpha)}{F(\alpha)}-\left(\frac{F'(\alpha)}{F(\alpha)}\right)^2$, so $C\geq K.$
Easy computations give that $F(\alpha)= n\omega_nH(\alpha)$ and so $K=\sum_{k=0}^{n-1}\frac{1}{(\alpha+k)^2}$.
\endproof

\begin{proposition}[Sub-exponential laws]
The probability measure $d\mu(x)=\frac{1}{Z}e^{-|x|^p}\,dx$ on $\R^n$ with $p\in(0,1)$ verifies the weighted Poincar\'e inequality
\[\Var_\mu(f)\leq C_{\textrm{opt}}\int |\nabla f|^2 |x|^{2(1-p)}\,d\mu(x),\]
where the optimal constant $C_{\textrm{opt}}$ is such that
\[\frac{n}{p^3}\leq C_{\textrm{opt}}\leq 12\frac{n}{p^3}+\frac{n+p}{p^4}.\]
\end{proposition}

\begin{remark}
As for the Cauchy law, letting $p$ go to 1 leads to
$$
\Var_\nu(f) \leq (13 n +1) \int |\nabla f|^2\,d\nu
$$
with $d\nu(x)=(1/Z)e^{-|x|} dx$. Again this recover
(with $13 n+1$ instead of $13n$) one particular result of Bobkov \cite{bobsphere}.
\end{remark}

\proof
We mimic the proof of the preceding example. Let $\psi(r)=\frac{1}{p}r^{p}$, $r\geq 0$ and define $\nu$ as the image of $\mu$ under the radial map $S(x)=\psi(|x|)\frac{x}{|x|}.$ Easy calculations give that the radial part of $\nu$ has density $\rho_\nu$ defined by
\[\rho_\nu(r)=\frac{n\omega_n}{Z}\left(\beta u\right)^{\frac{n-p}{p}}e^{-p u}.\]
It is clearly a log-concave function on $[0,+\infty).$ Let us compute the constant $C_\nu$ appearing in Theorem \ref{bobsphere}. One has
\[\int r\rho_\nu(r)\,dr=\frac{1}{p}\frac{\Gamma(\frac{n}{p}+1)}{\Gamma(\frac{n}{p})}=\frac{n}{p^2},\]
and
\[\int r^2\rho_\nu(r)\,dr=\frac{1}{p^2}\frac{\Gamma(\frac{n}{p}+2)}{\Gamma(\frac{n}{p})}=\frac{n(n+p)}{p^4}.\]
Consequently,
\[C_\nu=12\frac{n}{p^3}+\frac{n+p}{p^4}.\]
Now suppose that there is some $C$ such that $\Var_\mu(f)\leq C\int |\nabla f|^2|x|^{2(1-p)}\,d\mu(x)$ holds for all $f$. To prove that $C\geq \frac{n}{p^3}$, we will test this inequality on the functions $f_a(x)=e^{-a|x|^p}$, $a>0$.
Letting $G(t)=\int e^{-t|x|^p}\,d\mu(x)$, we arrive at the relation
\[C\geq \frac{1}{\beta^2a^2}\frac{G(1)G(2a+1)-G(a+1)^2}{G(1)G(2a+1)}, \qquad \forall a>0.\]
Letting $a\to0$, one obtains $C\geq \frac{1}{p^2}\left[\frac{G''(1)}{G(1)}-\left(\frac{G'(1)}{G(1)}\right)^2\right].$
The change of variable formula immediately yields $G(t)=t^{-\frac{n}{p}}G(1)$, and so $C\geq \frac{1}{p^2}\left[\frac{n(n+p)}{p^2}-\left(\frac{n}{p}\right)^2\right],$ which achieves the proof.
\endproof

\section{Links with weak Poincar\'e inequalities.} \label{sec:linkswp}

In this section we deal with weak Poincar\'e inequalities and work under the general setting of Section \ref{sec2}.
One says that a probability measure $\mu$ verifies the weak Poincar\'e inequality if for all $f\in \mathcal{A}$,
$$\Var_\mu(f) \, \leq \, \beta(s) \, \int \, \Gamma(f) \, d\mu \, + \,  \, s \,
\Osc_\mu(f)^2,\qquad \forall s\in (0,1/4),$$
where $\beta:(0,1/4)\to\R^+$ is a non-increasing function. Note that the limitation $s \in (0,1/4)$ comes from the bound
$\Var_\mu(f) \leq \Osc_\mu(f)^2/4$.

Weak Poincar\'e inequalities were introduced by R\"{o}ckner and Wang in \cite{r-w}.
In the symmetric case, they describe the decay of the semi-group $P_t$ associated to $L$ (see \cite{r-w,BCG}). Namely for all bounded centered function $f$, there exists $\psi(t)$ tending to zero at infinity such that $\| P_tf\|_{L_2(\mu)}\le \psi(t) \|f\|_\infty$.

They found another application in concentration of measure phenomenon for sub-exponential laws in \cite[Thm 5.1]{BCR2}. The approach proposed in \cite{BCR2} to derive weak Poincar\'e inequalities was based on capacity-measure arguments (following \cite{barthe-roberto}). In this section, we give alternative arguments. One is based on converse Poincar\'e inequalities.
This implies that weak Poincar\'e inequalities can be derived directly from the $\phi$-Lyapunov function strategy, using Theorem \ref{thmcwp}.
The second approach is based on a direct implication of weak Poincar\'e inequalities from weak Cheeger inequalities.
In turn, one can use either (the mass-transport technique of) Theorem \ref{th:weakcheeger} in order to get precise bounds for measures
on $\mathbb{R}^n$ which are tensor product of a measure on $\mathbb{R}$, or (via $\phi$-Lyapunov functions) Theorem \ref{thmweakC}.


Converse Poincar\'e inequalities imply weak Poincar\'e inequalities as shown in the following Theorem.
\begin{theorem}\label{thmweakP}
Assume that $\mu$ satisfies a converse Poincar\'e inequality $$\inf_c \, \int (g-c)^2 \, \omega \,
d\mu \le C \, \int \, \Gamma(g) \, d\mu \, $$ for some non-negative weight $\omega$, such
that $\bar{\omega}=\int \omega d\mu < +\infty$. Define $F(u)= \mu(\omega < u)$ and
$G(s)=F^{-1}(s):= \inf \{u ; \mu(\omega \leq u) > s\}$ for $s< 1$.

Then, for all $f\in \mathcal{A}$,
$$\Var_\mu(f) \, \leq \, \frac{C}{G(s)} \, \int \, \Gamma(f) \, d\mu \, + \,  \, s \Osc_\mu(f)^2,\qquad \forall s\in (0,1/4) .$$
\end{theorem}
\begin{proof}
The proof follows the same line of reasoning as the one of Theorem \ref{thmweakC}.
\end{proof}


Weak Poincar\'e inequalities are also implied by weak Cheeger inequalities as stated in the following Lemma.
The proof of the Lemma is a little bit more tricky than the usual one from Cheeger to
Poincar\'e. We give it for completeness.

\begin{lemma} \label{lem:weak}
Let $\mu$ be a probability measure and $\beta : \R^+ \to \R^+$. Assume that for any $f \in \mathcal{A}$ it holds
$$
\int |f - m|\, d\mu \leq \beta(s) \int \sqrt{\Gamma (f)}\, d\mu + s \Osc (f) \qquad \forall s \in (0,1)
$$
where $m$ is a median of $f$ under $\mu$. Then, any $f \in \mathcal{A}$
satisfies
\begin{equation} \label{eq:weakpoincare}
 \mathrm{Var}_\mu (f) \leq 4 \beta\left( \frac{s}{2} \right)^2 \int \Gamma(f) d\mu + s \Osc (f)^2
\qquad \forall s \in (0,1/4) .
\end{equation}
\end{lemma}

\begin{proof}
let $f \in \mathcal{A}$. Assume that $0$ is a median of $f$ and by homogeneity
of \eqref{eq:weakpoincare} that $\Osc (f) = 1$ (which implies in turn that $\|f\|_\infty \leq 1$).
Let $m$ be a median of $f^2$. Applying the weak Cheeger inequality to $f^2$, using the definition of the median
and the chain rule formula, we obtain
$$
\int f^2\, d\mu \leq \int |f^2 - m|\, d\mu \leq 2 \beta(s) \int |f| \sqrt{\Gamma(f)}\, d\mu + s \Osc (f^2)
\qquad \forall s\in (0,1) .
$$
Since $\|f\|_\infty \leq 1$ and $\Osc (f)=1$, one has  $\Osc (f^2) \leq 2$.
 Hence, by the Cauchy-Schwarz inequality, we have
$$
\int f^2\, d\mu
 \leq
2 \beta(s) \left( \int \Gamma(f)\, d\mu \right)^\frac{1}{2} \left( \int |f|^2\, d\mu
\right)^\frac{1}{2}  + 2s \qquad \forall s\in (0,1)  .
$$
Hence,
$$
\left( \int f^2\, d\mu \right)^\frac{1}{2}   \leq \beta(s) \left( \int \Gamma(f)\, d\mu
\right)^\frac{1}{2} + \left( \beta(s)^2 \int \Gamma(f)\, d\mu + s \right)^\frac{1}{2} .
$$
Since $\mathrm{Var}_\mu(f) \leq \int f^2\, d\mu$, we finally get
$$
\mathrm{Var}_\mu(f) \leq 4 \beta(s)^2 \int \Gamma(f)\, d\mu + 2s \qquad \forall s\in (0,1)
$$
which is the expected result.
\end{proof}

Two examples follow.

\begin{proposition}[Cauchy type laws]\label{prop:cauchyweakP}
Let $d\mu(x)= V^{-(n+\alpha)}(x) dx$ with $V$ convex on $\mathbb{R}^n$ and $\alpha>0$. Recall that $\kappa=-1/\alpha$.
Then there exists a constant $C>0$ such that for all smooth enough $f : \mathbb{R}^n \to \mathbb{R}$,
$$
\Var_\mu(f) \, \leq \, C s^{2\kappa} \, \int \, |\nabla f|^2 \, d\mu \, + \,  \, s \Osc_\mu(f)^2,\qquad \forall s\in (0,1/4) .
$$
\end{proposition}

\begin{proof}
The proof is a direct consequence of Proposition \ref{prop:convcauchy} together with Lemma \ref{lem:weak} above.
\end{proof}

\begin{remark}
For the generalized Cauchy distribution $d\mu(x)=c_\beta \, (1+|x|)^{-(n+\alpha)}$,
this result is optimal for $n=1$ and was shown in \cite{r-w} (see also \cite[Example 2.5]{BCR2}). For $n\geq 2$
the result obtained in \cite{r-w} is no more optimal. In \cite{BCG}, a weak Poincar\'e inequality is proved
in any dimension with rate function $\beta(s) \leq c(p) \, s^{2 p}$ for any $p<\kappa$.
Here we finally get the optimal rate. Note however that the constant $C$ may depend on $n$.
\hfill $\diamondsuit$
\end{remark}

\begin{proposition}[Sub exponential type laws] \label{prop:expweakP}
Let $d \mu= (1/Z_p) \, e^{- V^p}$
for some positive convex function $V$ on $\mathbb{R}^n$ and $p \in (0,1)$.
Then there exists $C > 0$ such that for all $f$
$$
\Var_{\mu}(f) \, \leq \, C \left( \log\left( \frac{1}{s} \right) \right)^{2(\frac{1}{p}-1)} \, \int \, |\nabla f|^2  \, d\mu \,
+ \,  \, s \Osc_{\mu}(f)^2,\qquad \forall s\in (0,1/4) .
$$
\end{proposition}

\begin{proof}
The proof is a direct consequence of Proposition \ref{prop:weakexp} together with Lemma \ref{lem:weak} above.
\end{proof}

By Lemma \ref{lem:weak} above, we see that weak Poincar\'e inequalities can be derived from mass-transport arguments using Theorem \ref{th:weakcheeger}.
This is stated in the next Corollary.

\begin{corollary} \label{cor:weakpoincare}
Let $\mu$ be a symmetric probability measure on $\R$ absolutely continuous with respect to the
Lebesgue measure. Assume that $\log \bF$ is convex on $\R^+$. Then, for any $n$, every function
$f: \R^n \to \R$ smooth enough satisfies
\begin{equation} \label{eq:weakpoincare2}
\Var_{\mu^n}(f) \leq \kappa_1^2\frac{s^2}{J_\mu(s/2)^2}\int |\nabla f|\, d\mu^n + 2\kappa_2ns \Osc(f)^2, \qquad \forall s>0.
\end{equation}
with $\kappa_1=2\sqrt{6}$ and $\kappa_2=2(1+2\sqrt{6})$.
\end{corollary}
\proof
Applying Lemma \ref{lem:weak} to $\mu^n$ together with Theorem \ref{th:weakcheeger} immediately yields the result.
\endproof

We illustrate this Corollary on two examples.

\begin{proposition}[Cauchy distributions] \label{prop:cauchyweakP2}
Consider $d\mu(x)= \frac{\alpha}{2(1+|x|)^{1+\alpha}}\,dx$ on $\R$, with $\alpha>0$. Then,
there is a constant $C$ depending only on $\alpha$ such that for all $n\geq 1$
\[
\Var_{\mu^n}(f)\leq C \left( \frac{n}{s} \right)^\frac{2}{\alpha} \int |\nabla f|^2\,d\mu^n
+ s\Osc_{\mu^n}(f)^2,\qquad \forall s\in (0,1/4).
\]
\end{proposition}

\begin{proof}
Since $J_{m_\alpha}(t)=\alpha 2^{1/\alpha}t^{1+1/\alpha}$ for $t\in(0,1/2)$, by Corollary
\ref{cor:weakpoincare}, on $\R^n$, $\mu^n$ satisfies a weak Poincar\'e inequality with rate
function $\beta(s) = C \left( \frac{n}{s} \right)^\frac{2}{\alpha}$, $s \in (0,
\frac{1}{4})$.
\end{proof}

\begin{proposition}[Sub-exponential law] \label{prop:expweakP2}
Consider the probability measure $\mu$ on $\R$, with density
$Z^{-1}e^{-|x|^p}$, $p \in (0,1]$. Then, there is a constant $C$ depending only on $p$ such that for all $n\geq 1$ \[\Var_{\mu^n}(f)\leq C\left( \log\left( \frac{n}{s} \right)
\right)^{2(\frac{1}{p}-1)}\int |\nabla f|^2\,d\mu^n + s\Osc_{\mu^n}(f)^2,\qquad \forall s\in (0,1/4). \]
\end{proposition}

\begin{proof}
By Corollary \ref{cor:iso2}, $J_{\mu}(t)$ is, up to a constant, greater than or equal to $t\left( \log(1/t) \right)^{1-\frac 1p}$ for $t \in
[0,1/2]$. Hence, by Corollary \ref{cor:weakpoincare}, $\mu^n$ satisfies a weak Poincar\'e
inequality on $\R^n$, with the rate function $\beta(s) = C \left( \log\left( \frac{n}{s} \right)
\right)^{2(\frac{1}{p}-1)}$, $s \in (0, \frac{1}{4})$.
\end{proof}

\begin{remark}
 The two previous results recover the results of \cite{BCR2}.
 Note the difference between the results of Proposition \ref{prop:expweakP} (applied to $V(x)=|x|$) and Proposition \ref{prop:expweakP2}.
This is mainly due to the fact that Proposition \ref{prop:expweakP} holds in great generality, while Proposition \ref{prop:expweakP2}
deals with a very specific distribution. The same remark applies to Propositions  \ref{prop:cauchyweakP} and \ref{prop:cauchyweakP2}
since in the setting of Proposition \ref{prop:cauchyweakP2}, $2/\alpha=-2\kappa$.

However, it is possible to recover the results of Proposition \ref{prop:expweakP2} (resp. Propositions  \ref{prop:cauchyweakP2})
applying Proposition \ref{prop:expweakP} (resp. Propositions  \ref{prop:cauchyweakP}) to the sub-exponential (resp. Cauchy) measure on $\mathbb{R}$ and
then to use the tensorization property \cite[Theorem 3.1]{BCR2}.
\hfill $\diamondsuit$
\end{remark}

\begin{remark}
According to an argument of Talagrand (recalled in the introduction), if for all $k$, $\mu^k$ satisfies the same
concentration property as $\mu$, then the tail distribution of $\mu$ is at most exponential. So no
heavy tails measure can satisfy a dimension-free concentration property.The concentration properties of heavy tailed measure are thus particularly intersting to study, and
in particular the dimension dependence of the result. The first results in this direction using weak Poincar\'e inequalities were done in \cite{BCR2}. As converse Poincar\'e inequalities plus control of the tail of the weight lead to  weak Poincar\'e inequality, and thus concentration, it is interesting to remark that
in Theorem 4.1 and Corollary 4.2 in \cite{BLweight}, Bobkov and Ledoux proved that if a weighted Poincar\'e
inequality holds, any 1-Lipschitz function with zero mean satisfies $$\parallel f\parallel_p \leq
\frac{D p}{\sqrt 2} \, \parallel \sqrt{1+\eta^2} \parallel_p$$ for all $p\geq 2$. It follows that
for all $t$ large enough ($t > D p e \parallel \sqrt{1+\eta^2}
\parallel_p$), $$\mu (|f|>t) \leq 2 \, \left(\frac{D \, p \, \parallel
\sqrt{1+\eta^2}
\parallel_p}{t}\right)^p \, .$$
Hence the concentration function is controlled by some moment of the weight. Dimension dependence
is hidden in this moment control. However if one is only interested in concentration properties, one could use directly weighted Poincar\'e inequalities.\hfill $\diamondsuit$
\end{remark}

\section{Appendix}
This appendix is devoted to the proofs of Proposition \ref{prop:is} and Corollary \ref{cor:iso2}.
Let us recall the first of these statements.\\

{\bf Proposition.} {\it Let $\Phi : \R^+ \rightarrow \R$ be a non-decreasing concave function
   satisfying $\Phi(x)/x \to 0$ as $x \to \infty$.
   Assume that in a neighborhood of $+\infty$ the function $\Phi$ is ${\mathcal C}^2$
   and there exists $\theta >1$ such that $\Phi^\theta$ is convex. Let $\mu_\Phi$ be defined in \eqref{eq:muphi}. Define $F_\mu$ and $J_\mu$ as in \eqref{eq:isopfonc}.

Then,
\begin{equation*}
\lim_{t \rightarrow 0} \frac{J_\mu(t)}{t \Phi' \circ \Phi^{-1}(\log \frac 1t)} = 1 .
\end{equation*}}

\begin{proof}[Proof of Proposition \ref{prop:is}]
The proof follows the line of \cite[Proposition 13]{BCR3}.
   By Point $(iii)$ of Lemma \ref{lem:tec} below,  $\Phi'$ never vanishes.
   Under our assumptions on $\Phi$ we have $F_\mu(y)=\int_{-\infty}^y
   Z_\Phi^{-1}e^{- \Phi(|x|)}dx \sim Z_\Phi^{-1} e^{-\Phi(|y|)} /
   \Phi'(|y|)$ when $y$ tends to $- \infty$.
  Thus using the   change of variable $y=F_\mu^{-1}(t)$, we get
\begin{eqnarray*}
\lim_{t \rightarrow 0} \frac{J_\mu(t)}{t \Phi' \circ \Phi^{-1}(\log \frac 1t)} & = & \lim_{y
\rightarrow - \infty}
\frac{e^{-\Phi(|y|)}}{Z_\Phi F_\mu(y) \Phi' \circ \Phi^{-1}(\log \frac{1}{F_\mu(y)})}\\
& = & \lim_{y \rightarrow - \infty} \frac{\Phi'(|y|)}{\Phi' \circ \Phi^{-1}(\log
\frac{1}{F_\mu(y)})} .
\end{eqnarray*}
By concavity of $\Phi$ we have $F_\mu(y) \geq Z_\Phi^{-1} e^{-\Phi(|y|)} / \Phi'(|y|)$ for all $y
\leq 0$. Hence, since $\lim_\infty \Phi' = 0$, we have $\log\frac{1}{F_\mu(y)} \leq \Phi(|y|)$
when $y \ll -1$.

Then, a Taylor expansion of $\Phi' \circ \Phi^{-1}$ between $\log\frac{1}{F_\mu(y)}$ and
$\Phi(|y|)$ gives
$$
\frac{\Phi' \circ \Phi^{-1}(\log \frac{1}{F_\mu(y)})}{ \Phi'(|y|)} = 1 +
\frac{1}{\Phi'(|y|)}\left(\log \frac{1}{F_\mu(y)} - \Phi(|y|) \right) \frac{\Phi'' \circ \Phi^{-1}
(c_y)}{\Phi' \circ \Phi^{-1} (c_y)}
$$
for some $c_y \in [\log\frac{1}{F_\mu(y)} , \infty)$.

For $y \ll -1$, we have
\begin{equation} \label{eq:enc}
\frac{e^{-\Phi(|y|)}}{Z_\Phi \Phi'(|y|)} \leq F_\mu(y) \leq 2\frac{e^{-\Phi(|y|)}}{Z_\Phi
\Phi'(|y|)}.
\end{equation}
Hence, using Point $(iii)$ of Lemma \ref{lem:tec} below,
\begin{eqnarray*} 
\left| \log \frac{1}{F_\mu(y)} - \Phi(|y|) \right| & = &
\Phi(|y|) - \log \frac{1}{F_\mu(y)} \\
& \leq &
\log\frac{2}{Z_\Phi} + \log\left( \frac{1}{\Phi'(|y|)} \right) \\
& \leq & \log\frac{2}{Z_\Phi} + c \log(|y|)
\end{eqnarray*}
for some constant $c$ and all $y \ll -1$.

On the other hand, when $\Phi^\theta$ is convex and $\mathcal C^2$, $(\Phi^\theta) ''$ is non
negative. This, together with Point $(i)$ of Lemma \ref{lem:tec}, lead to
$$
\left|\frac{\Phi''(x)}{\Phi'(x)}\right| = - \frac{\Phi''(x)}{\Phi'(x)} \leq
(\theta-1)\frac{\Phi'(x)}{\Phi(x)} \leq \frac{c'}{x}
$$
for some constant $c'$ and $x \gg 1$. It follows that
$$
\frac{\Phi'' \circ \Phi^{-1} (c_y)}{\Phi' \circ \Phi^{-1} (c_y)} \leq \frac{c'}{\Phi^{-1}
\left(\log \frac{1}{F_\mu(y)} \right)} .
$$
Now, by \eqref{eq:enc} and Point $(iii)$ and $(ii)$ of Lemma \ref{lem:tec}, we note that
\begin{eqnarray*}
\log \frac{1}{F_\mu(y)} & \geq &
\Phi(|y|) + \log \left(\frac{Z_\Phi}{2} \right)  + \log( \Phi'(|y|) ) \\
& \geq &
\Phi(|y|) + \log \left(\frac{Z_\Phi}{2} \right)  -c_3\log( |y|) ) \\
& \geq &
\Phi(|y|) + \log \left(\frac{Z_\Phi}{2} \right)  -\frac{c_3}{c_2}\log( \Phi(|y|)) ) \\
& \geq & \frac 12 \Phi(|y|)
\end{eqnarray*}
provided $y \ll -1$. In turn, by Point $(iv)$ of Lemma \ref{lem:tec},
$$
\frac{\Phi'' \circ \Phi^{-1} (c_y)}{\Phi' \circ \Phi^{-1} (c_y)} \leq \frac{c''}{|y|}
$$
for some constant $c''$.

All these computations together give
$$
\left|\frac{1}{\Phi'(|y|)}\left(\log \frac{1}{H(y)} - \Phi(|y|)
   \right) \frac{\Phi'' \circ \Phi^{-1} (c_y)}{\Phi' \circ \Phi^{-1}
     (c_y)} \right|
\leq c''\frac{\log\frac{2}{Z_\Phi} + c \log(|y|)}{|y|\Phi'(|y|)}
$$
which goes to $0$ as $y$ goes to $-\infty$ by Point $(i)$ and $(ii)$ of Lemma \ref{lem:tec}. This
ends the proof.
\end{proof}
\smallskip

\begin{lemma} \label{lem:tec}
Let $\Phi : \R^+ \rightarrow \R$ be an increasing concave function
   satisfying $\Phi(x)/x \to 0$ as $x \to \infty$.
   Assume that in a neighborhood of $+\infty$ the function $\Phi$ is ${\mathcal C}^2$
   and there exists $\theta >1$ such that $\Phi^\theta$ is convex.
   Assume that $\int e^{- \Phi(|x|)}dx < \infty$.  Then, there exist constants
    $c_1,c_3 > 1$, $c_2, c_4 \in (0,1)$ such that
for $x$ large enough,\\
$(i)$
$c_1^{-1} x \Phi'(x) \leq \Phi(x) \leq c_1 x \Phi'(x)$;\\
$(ii)$
$\Phi(x) \geq x^{c_2}$;\\
$(iii)$
$\Phi'(x) \geq x^{-c_3}$;\\
$(iv)$ $\frac{1}{2} \Phi(x) \geq \Phi(c_4 x)$.
\end{lemma}

\begin{proof}
Let $\widetilde \Phi=\Phi - \Phi(0)$. Then, in the large, $\widetilde \Phi$ is concave and
$(\widetilde \Phi)^\theta$ is convex. Hence, the slope functions $\widetilde \Phi(x)/x$ and
$(\widetilde \Phi)^\theta/x$ are non-increasing and non-decreasing respectively. In turn, for $x$
large enough,
$$
x \Phi'(x) = x \widetilde \Phi'(x) \leq \widetilde \Phi (x) \leq \theta x \widetilde \Phi'(x) =
\theta x \Phi'(x).
$$
This bound implies in particular that $x \Phi'(x) \to \infty$ as $x$ tends to infinity. Point
$(i)$ follows.

\smallskip

The second inequality in $(i)$ implies that for $x$ large enough,
\begin{equation} \label{eq:2}
\frac{\Phi'(x)}{\Phi(x)} \geq \frac{1}{c_1x} .
\end{equation}
Hence, for some $x_0$ large enough, integrating, we get
$$
\log\Phi(x) \geq \log \Phi(x_0) + \frac{1}{c_1} \left( \log(x) - \log(x_0) \right) \geq
\frac{1}{2c_1} \log(x) \quad \forall x \gg x_0.
$$
Point $(ii)$ follows.

\smallskip

Point $(iii)$ follows from the latter and Inequality \eqref{eq:2}.

\smallskip

Take $c=\exp\{1/c_1\}$. By Point $(i)$, we have for $x$ large enough
\begin{eqnarray*}
\Phi(cx) &=&
\Phi(x) + \int_x^{cx} \Phi'(t) dt \\
& \geq &
\Phi(x) + \int_x^{cx} \frac{\Phi(t)}{c_1t} dt \\
& \geq &
\Phi(x) \left( 1 + \int_x^{cx} \frac{1}{c_1t} \right) dt \\
& = & \Phi(x) \left(1 + \frac{\log c}{c_1} \right) = 2 \Phi(x) .
\end{eqnarray*}
Point $(iv)$ follows.
\end{proof}
\smallskip

Now let us recall the statement of Corollary \ref{cor:iso2}.\\

{\bf Corollary.} {\it Let $\Phi : \R^+ \rightarrow \R$ be a non-decreasing concave function
   satisfying $\Phi(x)/x \to 0$ as $x \to \infty$ and $\Phi(0) < \log 2$.
   Assume that in a neighborhood of $+\infty$ the function $\Phi$ is ${\mathcal C}^2$
   and there exists $\theta >1$ such that $\Phi^\theta$ is convex.
   Let $d \mu(x) = Z_\Phi^{-1}e^{- \Phi(|x|)}dx$ be a probability
   measure on $\R$. Then, $$
I_{\mu^n}(t) \geq c \min(t,1-t) \Phi' \circ \Phi^{-1}\left(\log \frac{n}{\min(t,1-t)}\right)
\qquad \forall t \in [0,1], \; \forall n
$$
for some constant $c>0$ independent on $n$.}

\begin{proof}[Proof of Corollary \ref{cor:iso2}.]
Since $\Phi$ is concave, $\log (1-F_\mu)$ is convex on $\R^+$. Applying Corollary
\ref{cor:isotransport} together with Proposition \ref{prop:is} lead to
$$
I_{\mu^n}(t) \geq c \min(t,1-t) \Phi' \circ \Phi^{-1} \left(\log \frac{n}{c'\min(t,1-t)} \right)
\qquad \forall t \in [0,1],\; \forall n
$$
for some constant $c>0$ and $c'> 1$ independent on $n$. It remains to prove that for all $t \in
[0,1/2]$,
$$
t \Phi' \circ \Phi^{-1} \left(\log \frac{n}{c't} \right) \geq c'' t \Phi' \circ \Phi^{-1}
\left(\log \frac{n}{t} \right)
$$
for some constant $c''>0$. For $t\leq 1/2$ we have $1/(c't) \leq (1/t)^C$ for some $C >1$. Hence,
since $\Phi' \circ \Phi^{-1}$ is non-increasing,
$$
\Phi' \circ \Phi^{-1}(\log \frac{n}{c't}) \geq \Phi' \circ \Phi^{-1}(C \log \frac nt).
$$
Now note that Point $(iv)$ of Lemma \ref{lem:tec} is equivalent to say $\Phi^{-1}(2x) \leq
\frac{1}{c_4} \Phi^{-1}(x)$ for $x$ large enough. Hence $\Phi^{-1}(Cx) \leq \left( \frac{1}{c_4}
\right)^{\lfloor \log_2 C\rfloor +1}  \Phi^{-1}(x)$. It follows that
$$
\Phi' \circ \Phi^{-1}(\log \frac{n}{c't}) \geq \Phi' \left( \left( \frac{1}{c_4} \right)^{\lfloor
\log_2 C\rfloor +1} \Phi^{-1}(\log \frac nt) \right)
$$
for $t$ small enough. Finally, Point $(i)$ and $(iv)$ of Lemma \ref{lem:tec} ensure that
$$
\Phi' \left( \frac{1}{c_4} x \right) \geq \frac{c_4}{c_1} \frac{\Phi \left( \frac{x}{c_4}
\right)}{x} \geq \frac{2 c_4}{c_1} \frac{\Phi \left( x \right)}{x} \geq \frac{2 c_4}{c_1^2}
\Phi'(x) .
$$
Hence
$$
t \Phi' \circ \Phi^{-1} \left(\log \frac{n}{c't} \right) \geq c'' t \Phi' \circ \Phi^{-1}
\left(\log \frac{n}{t} \right)
$$
for some constant $c''>0$ and $t$ small enough, say for $t \leq t_0$. The expected result follows
by continuity of $t \mapsto t \Phi' \circ \Phi^{-1}(\log \frac nt)/t \Phi' \circ \Phi^{-1}(\log
\frac{n}{c't})$ (on $[t_0, 1/2]$).
\end{proof}

\bibliographystyle{plain}
\bibliography{weightpoincareb}
\end{document}